\documentclass[12pt,letterpaper,reqno]{amsart}

\usepackage{times}

\usepackage[T1]{fontenc}

\usepackage{mathrsfs}

\usepackage{latexsym}

\usepackage[dvips]{graphics}

\usepackage{epsfig}

\usepackage{amsmath,amsfonts,amsthm,amssymb,amscd, hyperref}

\usepackage{verbatim}

\input amssym.def

\input amssym.tex

\usepackage{color}

\usepackage{hyperref}

\usepackage{url}

\newcommand\be{\begin{equation}}

\newcommand\ee{\end{equation}}

\newcommand\bea{\begin{eqnarray}}

\newcommand\eea{\end{eqnarray}}

\newcommand\bi{\begin{itemize}}

\newcommand\ei{\end{itemize}}

\newcommand\ben{\begin{enumerate}}

\newcommand\een{\end{enumerate}}

\newcommand\bc{\begin{center}}

\newcommand\ec{\end{center}}

\newcommand\ba{\begin{array}}

\newcommand\ea{\end{array}}

\newtheorem{thm}{Theorem}[section]

\newtheorem{conj}[thm]{Conjecture}

\newtheorem{cor}[thm]{Corollary}

\newtheorem{lem}[thm]{Lemma}

\newtheorem{prop}[thm]{Proposition}

\newtheorem{defi}[thm]{Definition}

\numberwithin{equation}{section}

\begin{document}

\title[Continued Fraction Digit Averages and Maclaurin's Inequalities]{Continued Fraction Digit Averages and Maclaurin's Inequalities}

\author[Cellarosi]{Francesco Cellarosi}\email{\textcolor{blue}{\href{fcellaro@illinois.edu}{fcellaro@illinois.edu}}}
\address{Department of Mathematics, University of Illinois at Urbana-Champaign, 1409 W Green Street, Urbana IL 61801}

\author[Hensley]{Doug Hensley}\email{\textcolor{blue}{\href{dahensley@suddenlink.net}{dahensley@suddenlink.net}}}\address{Department of Mathematics, Texas A\&M University, College Station, Texas 77843}

\author[Miller]{Steven J. Miller}\email{\textcolor{blue}{\href{mailto:sjm1@williams.edu}{sjm1@williams.edu}},  \textcolor{blue}{\href{Steven.Miller.MC.96@aya.yale.edu}{Steven.Miller.MC.96@aya.yale.edu}}}
\address{Department of Mathematics and Statistics, Williams College, Williamstown, MA 01267}

\author[Wellens]{Jake L. Wellens}\email{\textcolor{blue}{\href{jwellens@caltech.edu}{jwellens@caltech.edu}}}
\address{California Institute of Technology, Pasadena, CA 91126}

\subjclass[2000]{
11K50, %
26D05 
 (primary),
26D20, 
26D15, 
33C45  
 (secondary).
 }

\keywords{Continued fractions, metric theory of continued fractions, arithmetic mean, geometric mean, AM-GM inequality, Maclaurin's inequalities, phase transition, quadratic surds.
}

\date{\today}

\thanks{The first named author was partially supported by AMS-Simons Travel grant, the third named author by NSF grants DMS0970067 and DMS1265673, and the fourth named author by DMS0850577 and Williams College. We thank Iddo Ben-Ari and Keith Conrad for sharing the preprint \cite{BenAri-Conrad} with us, Xiang-Sheng Wang for mentioning formula \eqref{generalized-Laplace-Heine-formula} to us, and Harold G. Diamond for useful discussions.}

\begin{abstract}
A classical result of Khinchin says that for almost all real numbers $\alpha$, the geometric mean of the first $n$ digits $a_i(\alpha)$ in the continued fraction expansion of $\alpha$ converges to a number $K \approx 2.6854520\ldots$ (Khinchin's constant) as $n \to \infty$. On the other hand, for almost all $\alpha$, the arithmetic mean of the first $n$ continued fraction digits $a_i(\alpha)$ approaches infinity as $n \to \infty$. There is a sequence of refinements of the AM-GM inequality, Maclaurin's inequalities, relating the $1/k$\textsuperscript{th} powers of the $k$\textsuperscript{th} elementary symmetric means of $n$ numbers for $1 \leq k \leq n$. On the left end (when $k=n$) we have the geometric mean, and on the right end ($k=1$) we have the arithmetic mean.

We analyze what happens to the means of continued fraction digits of a typical real number in the limit as one moves $f(n)$ steps away from either extreme. We prove sufficient conditions on $f(n)$ to ensure divergence when one moves $f(n)$ steps away from the arithmetic mean and convergence when one moves $f(n)$ steps away from the geometric mean. We show for almost all $\alpha$ and appropriate $k$ as a function of $n$ that $S(\alpha,n,k)$ is on the order of $\log(n/k)$. For typical $\alpha$ we conjecture the behavior for $f(n)=cn$, $0<c<1$. We also study  the limiting behavior of such means for quadratic irrational $\alpha$, providing rigorous results, as well as numerically supported conjectures.
\end{abstract}

\maketitle

\setcounter{equation}{0}







\tableofcontents

\section{Introduction}

Each irrational number $\alpha \in (0,1)$ has a unique \emph{continued fraction expansion} of the form
\be \alpha \ = \  \dfrac{1}{a_1(\alpha) + \dfrac{1}{a_2(\alpha) + \dfrac{1}{ \ddots}}}, \ee
where the $a_i(\alpha) \in \mathbb{N}^+$ are called the continued fraction digits of $\alpha$.
In 1933, Khinchin \cite{Khinchin-book}  published the first fundamental results on the behavior of various averages of such digits. He showed that for functions $f(r) = O(r^{1/2 - \epsilon})$ as $r\to\infty$ the following equality holds for almost all $\alpha \in (0,1)$:

\be \lim_{n \to \infty} \dfrac{1}{n} \sum_{k=1}^{n} f(a_k(\alpha)) \ = \  \sum_{r=1}^{\infty} f(r) \log_{2}{\left(1+ \frac{1}{r(r+2)}\right)}. \ee
In particular, when we choose $f(r) = \ln{r}$ and exponentiate both sides, we find that

\be \lim_{n \to \infty} \left(a_1(\alpha) \cdots a_n(\alpha)\right)^{1/n} \ = \  \prod_{r=1}^{\infty} \left( 1+ \dfrac{1}{r(r+2)}\right)^{\log_{2}{r}} \ =:\ K_0.\label{K_0}\ee The constant $K_0\approx2.6854520\dots$ is known as \emph{Khinchin's constant}. See \cite{Bailey-Borwein-Crandall} for several series representations and numerical algorithms to compute $K_0$.
Khinchin \cite{Khinchin-book} also proved
that if $\{\phi(n)\}$ is a sequence of natural numbers, then for almost all $\alpha \in (0,1)$ 
 \be a_n(\alpha)\ >\ \phi(n) \:\text{ for at most finitely many } n\:\ \ \iff\ \  \sum_{n=1}^{\infty} \frac{1}{\phi(n)} \ < \  \infty. \ee
This implies, in particular, that for almost all $\alpha$ the inequality
\be a_n(\alpha)\ >\ n\log{n} \ee
holds infinitely often, and thus
\be \dfrac{a_1(\alpha) + \cdots + a_n(\alpha)}{n}\ >\ \log{n} \label{arithmetic-means-exceeds-logn-i.o.}\ee
for infinitely many $n$. So, for a typical continued fraction, the geometric mean of the digits converges while the arithmetic mean diverges to infinity.
This fact is a particular manifestation of the classical inequality relating arithmetic and geometric means for sequences of nonnegative real numbers.

The geometric and arithmetic means are actually the endpoints of a chain of inequalities 
relating elementary symmetric means. More precisely, let the \emph{$k$\textsuperscript{th} elementary symmetric mean} of an $n$-tuple $X = (x_1, \dots, x_n)$  be
 \be S(X, n, k)\ := \ \frac{\displaystyle\sum_{1\leq i_1 < \cdots < i_k \leq n} x_{i_1}x_{i_2}\cdots x_{i_k}}{\dbinom{n}{k}}. \ee
\emph{Maclaurin's Inequalities} \cite{BenAri-Conrad,Maclaurin-1729} state that, when the entries of $X$ are positive, we have
\be S(X, n, 1)^{1/1} \ \ge \  S(X, n, 2)^{1/2} \ \ge \  \cdots \ \ge \  S(X, n, n)^{1/n}, \label{Maclaurin-inequalities}\ee
and the equality signs hold if and only if $x_1=\cdots=x_n$. The standard proof of \eqref{Maclaurin-inequalities} uses Newton's inequality (see \cite{Beckenbach-Bellman}). Notice that $S(X,n,1)^{1/1}=\frac{1}{n}\left(x_1+\cdots+ x_n\right)$ is the arithmetic mean and $S(X,n,n)^{1/n}=\left( x_1\cdots x_n\right)^{1/n}$ is the geometric mean of the entries of $X$.

In view of Khinchin's results discussed above, it is natural to consider the case when $X = (a_1(\alpha), \dots, a_n(\alpha))$ is a tuple of continued fraction digits, and to write $S(\alpha, n, k)$ instead of $S(X, n, k)$. Khinchin's results say that for almost all $\alpha$,
\be S(\alpha, n,1)^{1/1}\to\infty \hspace{.3cm}\mbox{ and }\hspace{.3cm} S(\alpha, n, n)^{1/n} \to K_0\label{Khinchin-rephrased}\ee
as $n \to \infty$. 
In this paper we investigate the behavior of the intermediate means $S(\alpha,n,k)^{1/k}$ as $n\to\infty$, when $1\leq k\leq n$ is a function of $n$. In other words, we attempt to characterize the potential phase transition in the limit behavior of the means $S(\alpha,n,k)^{1/k}$.\\ \

\emph{We always assume that if the function $k=k(n)$ is not integer-valued, then $S(\alpha,n,k(n))^{1/k(n)}$ $=$ $S(\alpha,n,\lceil k(n)\rceil)^{1/\lceil k(n)\rceil}$, where $\lceil\cdot\rceil$ denotes the ceiling function.}\\ \

Our main results are the following theorems, which can be seen as generalizations of Khinchin's classical results \eqref{Khinchin-rephrased}.

\begin{thm}\label{thm1}
Let $f(n)$ be an arithmetic function such that $f(n)=o(\log{\log{n}})$ as $n\to\infty$. Then, for almost all $\alpha$, 
\be \lim_{n \to \infty} S(\alpha, n, f(n))^{1/f(n)} \ = \  \infty.\label{statement-thm1}\ee
\end{thm}

\begin{thm}\label{thm2}
Let $f(n)$ be an arithmetic function such that $f(n)=o(n)$ as $n\to\infty$. Then, for almost all $\alpha$, 
\be \lim_{n \to \infty} S(\alpha, n, n-f(n))^{1/(n-f(n))} \ = \  K_0. \label{statement-thm2}\ee
\end{thm}

\begin{thm}\label{thm:maindough}
There exist absolute, effectively computable positive constants $N$, $C_1$,  $C_2$, and $R$ with $C_1<C_2$ and $R>1$ such that for all $n\ge N$, for all $k$ with $n^{3/4}\le k\le n/R$, and for all $\alpha$ in a set $G\subset[0,1]$ of measure at least $1-n^{-4}$, \begin{equation} C_1\log(n/k) \ \le\ S(\alpha,n,k)^{1/k}\ \le\ C_2\log(n/k). \end{equation}
\end{thm}

Theorem \ref{thm:maindough} immediately yields a strengthened version of Theorem \ref{thm1}.

\begin{cor}\label{cor:maindough} If $f(n)=o(n)$, then with probability 1, \be \lim_{n\rightarrow\infty}S(\alpha,n,f(n))^{1/f(n)}\ = \ \infty. \ee
\end{cor}

Theorems \ref{thm1} and \ref{thm2} do not include the case of $k=c n$, $0<c<1$. In fact, for means of the type $S(\alpha,n,cn)^{1/cn}$ we can only provide bounds for the limit superior (Proposition \ref{bound-limsup-babyversion} and Theorem \ref{bound-limsup-improvedversion}). On the other hand, assuming that $\lim_{n\to\infty}S(\alpha,n,cn)^{1/cn}$ exists for almost every $\alpha$ (Conjecture \ref{hyp1}), we can show that the limit is a continuous function of $c$ (Theorem \ref{F-is-continuous}). We also conjecture an explicit formula for the almost sure limit (Conjecture \ref{hyp2}).

Theorem \ref{thm:maindough} tells us that the correct scale for $S(\alpha,n,k)^{1/k}$ is $\log(n/k)$; a topic for future research is to localize this quantity more precisely. The first two theorems are proved by a direct analysis of the desired expressions, while the third theorem is proved by considering related systems with similar digits which are more amenable to bounding. While a strengthened version of Theorem \ref{thm1} follows immediately from Theorem \ref{thm:maindough}, we have chosen to present an independent proof of this weaker case as the argument is significantly more elementary and requires less technical machinery.

Since \eqref{Khinchin-rephrased}-\eqref{statement-thm2} only hold for a \emph{typical} $\alpha$ (in the sense of measure), it is natural to study what happens to $S(\alpha,n,k)^{1/k}$  as $n\to\infty$ for \emph{particular} $\alpha$ (see Appendix \ref{sec:computationalimprovements} for a discussion of ways to speed up the computations for general $\alpha$). For example $\alpha=\sqrt3-1=[1,2,1,2,1,2,\ldots]$ satisfies
\be
\lim_{n\to\infty}S(\alpha,n,1)^{1/1}\ = \ \frac{3}{2}\ \neq\ \infty,\ \ \ \ \ \lim_{n\to\infty}S(\alpha,n,n)^{1/n}\ = \ \sqrt{2}\ \neq\ K_0,
\ee
and it is natural to ask whether $\lim_{n\to\infty}S(\alpha,n,cn)^{1/cn}$ for $0<c<1$ exists, and what its value is. When $c=1/2$ we prove that for $\alpha$ with a (pre)periodic continued fraction expansion with period 2 the limit $\lim_{n\to\infty}S(\alpha,2n,n)^{1/n}$ exists and we provide an explicit formula for it (see Lemma \ref{lemma-monotonicity}). This is a non-trivial fact following from an asymptotic formula for Legendre polynomials. For other values of $c$ the same result is expected to be true and is related to asymptotic properties of hypergeometric functions. This is not surprising, given the recent results connecting Maclaurin's inequalities with the Bernoulli inequality \cite{BenAri-Conrad} and the Bernoulli inequality with hypergeometric functions  \cite{Klen-Monojlovic-Simic-Vuorinen}. We perform a numerical analysis and we are able to  conjecture that the limit exists for all $L$-periodic $\alpha$ and all $0<c\leq 1$ (Conjecture \ref{conj-periodic-c}).

Assuming Conjecture \ref{conj-periodic-c}, we are able to give an explicit construction that approximates $S(\alpha,n,cn)^{1/cn}$ for typical $\alpha$'s with the same average for a periodic sequence of digits, with increasing period. This construction allows us to provide a strengthening of Theorem \ref{thm1} where, assuming Conjectures \ref{hyp1} and \ref{conj-periodic-c}, the assumption $o(\log\log n)$ can be replaced by $o(n)$ (Theorem \ref{thm2.7}).


\section{The proof of Theorems \ref{thm1} and \ref{thm2}}
We begin with a useful strengthening of Maclaurin's inequalities due to C. Niculescu.

\begin{prop}[\cite{Niculescu}, Theorem 2.1 therein]\label{prop-Niculescu}
If $X$ is any $n$-tuple of positive real numbers, then for any $0 < t < 1$ and any $j, k \in \mathbb{N}$ such that $t j + (1-t) k \in \{1, \dots, n\}$, we have \be S(X, n, t j + (1-t) k) \ \ge \  S(X, n, j)^{t} \cdot S(X, n, k)^{1-t} \,.\ee \\
\end{prop}

The next lemma shows that if the limit $\lim_{n\to\infty}S(X,n, k(n))^{1/k(n)}$ exists, then it is robust under small perturbations of $k(n)$.

\begin{lem}\label{lemma-little-o-perturbation}  Let $X$ be a sequence of positive real numbers. Suppose $\lim_{n\to\infty}S(X,n, k(n))^{1/k(n)}$ exists. Then, for any $f(n) = o(k(n))$ as $n\to\infty$, we have
\be \lim_{n\to\infty}S(X,n, k(n)+f(n))^{1/(k(n)+f(n))} \ = \  \lim_{n\to\infty}S(n, k(n))^{1/k(n)}.\ee
\end{lem}

\begin{proof}
First assume that $f(n) \geq 0$ for large enough $n$. For display purposes we write $k$ and $f$ for $k(n)$ and $f(n)$ below. From Newton's inequalities and Maclaurin's inequalities, we get
\be \left(S(X,n, k)^{1/k}\right)^{\frac{k}{k+f}}\ =\ S(X,n, k)^{1/(k+f)}\ \leq\ S(X,n, k+f)^{1/(k+f)}\ \leq\ S(X,n, k)^{1/k}.\ee
Taking $n \to \infty$, we see both the left and right ends tend to the same limit, and so then must the middle term. A similar argument works for $f(n) < 0$.
\end{proof}

We can now prove our first main theorem.

\begin{proof}[Proof of Theorem \ref{thm1}] Notice each entry of $\alpha$ is at least 1.
Let $f(n) = o(\log{\log{n}})$. Set $t=1/2$ and $(j, k) = (1, 2f(n) -1)$, so that $tj+(1-t)k = f(n)$. Then Proposition \ref{prop-Niculescu} yields
\be S(\alpha, n, f(n)) \ \ge \  \sqrt{S(\alpha, n, 1)\cdot S(\alpha, n, 2f(n) - 1)}\ >\ \sqrt{S(\alpha, n, 1)}, \ee
whereupon squaring both sides and raising to the power $1/f(n)$, we get
\be S(\alpha, n, f(n))^{2/f(n)} \ \ge \  S(\alpha, n, 1)^{1/f(n)}.\label{pf1}\ee
It follows from \eqref{arithmetic-means-exceeds-logn-i.o.} that,
for every function $g(n)=o(\log n)$ as $n\to\infty$, \be \lim_{n \to \infty} \frac{S(\alpha, n, 1)}{g(n)} \ = \  +\infty \ee
for almost all $\alpha$. Let $g(n) = \log{n}/\log{\log{n}}$. Taking logarithms, we have for sufficiently large $n$
\be\log\left( S(\alpha,n,1)^{1/f(n)}\right)\  > \ \frac{\log g(n)}{f(n)}\ >\ \frac{ \log{\log{n}}}{2f(n)}\label{pf2}.\ee
The assumption  $f(n) = o(\log{\log{n}})$, along with \eqref{pf1} and \eqref{pf2}, give the desired divergence.
\end{proof}

\begin{prop}\label{bound-limsup-babyversion}
For any constant $0 < c < 1$, and for almost all $\alpha$, we have
\be K_0 \ \le \  \limsup_{n \to \infty}S(\alpha, n, cn)^{1/cn} \ \le \   K_0^{1/c} \ < \  \infty.\label{statement-bound-limsup-babyversion}\ee
\end{prop}

\begin{proof}
 We have
\be S(\alpha, n, cn)^{1/cn} \ = \  \left( \prod_{i=1}^{n} a_i(\alpha)^{1/n} \right)^{n/cn} \left( \dfrac{ \displaystyle\sum_{i_1 < \cdots < i_{(1-c)n} \leq n} 1/(a_{i_1}(\alpha) \cdots a_{i_{(1-c)n}}(\alpha))}{\dbinom{n}{cn}}\right)^{1/cn}. \label{pf3}\ee
Note that the first factor is just the geometric mean, raised to the $1/c$ power, so this converges almost everywhere to $K_0^{1/c}$. Since each term in the sum is bounded above by 1, and there are exactly $\binom{n}{cn}$ of them, the second factor is bounded above by 1 and thus the whole limit superior is bounded above by $K_0^{1/c}$ almost everywhere.  However, Maclaurin's inequalities \eqref{Maclaurin-inequalities} tell us that almost everywhere $S(\alpha, n, cn)^{1/cn}$ must be at least $K_0/(1+\epsilon)$ for sufficiently large $n$ and any $\epsilon > 0$. Thus, for almost all $\alpha$,
\be K_0 \ \le \  \limsup_{n \to \infty} S(\alpha, n, cn)^{1/cn} \ \le \  K_0^{1/c}.\ee
\end{proof}

Theorem \ref{thm2} is a corollary of Proposition \ref{bound-limsup-babyversion}.

\begin{proof}[Proof of Theorem \ref{thm2}]
Since $f(n)=o(n)$, for any $c < 1$ we have for sufficiently large $n$ that $n \geq n-f(n) > cn$. Thus by
\eqref{Maclaurin-inequalities} and \eqref{statement-bound-limsup-babyversion},
\begin{align} K_0 &\ = \  \lim_{n \to \infty} S(\alpha, n, n)^{1/n} \ \le \  \lim_{n \to \infty} S(\alpha, n, n-f(n))^{1/(n-f(n))} \nonumber\\
&\ \le \  \limsup_{n \to \infty} S(\alpha, n, cn)^{1/cn } \ \le \  K_0^{1/c}. \end{align}
Since $c < 1$ was arbitrary, we can take $c \to 1$, which proves the desired result.
\end{proof}

\section{The linear regime $k=cn$}

We already gave upper and lower bounds for $\limsup_{n\to\infty}S(\alpha,n,cn)^{1/cn}$ in Proposition \ref{bound-limsup-babyversion}. Here we provide an improvement of the upper bound, which requires a little more notation.

First, let us recall another classical result concerning H\"older means for continued fraction digits. For any real non-zero $p<1$ the mean
\be \left(\frac{1}{n}\sum_{i=1}^n a_i^p\right)^{1/p}\ee
converges for almost every $\alpha$ as $n\to\infty$ to the constant
\be K_p\ = \ \left(\sum_{r=1}^\infty -r^p\log_2\left(1-\frac{1}{(r+1)^2}\right)\right)^{1/p}\label{p-Holder-mean}.\ee
A proof of this fact for $p<1/2$ can be found in \cite{Khinchin-book}; for $p<1$ see \cite{Ryll-Nardzewski}. Other remarkable formulas for $K_p$
are proven in \cite{Bailey-Borwein-Crandall}. The reason why we denoted \eqref{K_0} by $K_0$ is that $\lim_{p\to0} K_p=K_0$.
Notice that, for $p=-1$, \eqref{p-Holder-mean} gives the almost everywhere value\footnote{An interesting example for which the harmonic mean exists and differs from $K_{-1}$ is $e-2=[1,2,1,1,4,1,1,6,1,1,8,1,1,10,\ldots]$, which has harmonic mean $3/2$. Furthermore, notice that its geometric mean is divergent.}  of the harmonic mean
\be\lim_{n\to\infty}\frac{n}{\frac{1}{a_1}+\cdots+\frac{1}{a_n}}\ = \ K_{-1}\ \approx\ 1.74540566240\dots.\label{harmonic-mean}\ee

Since we want to improve Proposition \ref{bound-limsup-babyversion}, we are interested in the  behavior of the second factor of \eqref{pf3}. It is thus useful to define the inverse means
\be R(\alpha, n, k) \ :=\  \left( \dfrac{\displaystyle \sum_{1\leq i_1 < \cdots < i_k \leq n} (a_{i_1}(\alpha)\cdots a_{i_k}(\alpha))^{-1}}{\dbinom{n}{k}} \right).\label{inverse-means}\ee
Observe that $R(\alpha, n,k)=S(X,n,k)$ where $X=(x_i)_{i\geq 1}$ and $x_i=1/a_i$.
Notice that \eqref{harmonic-mean} reads as, for almost every $\alpha$,
\be \lim_{n \to \infty} R(\alpha, n, 1) \ = \  \frac{1}{K_{-1}}\ \approx\ 0.572937\dots\label{harmonic-mean2}.\ee

\begin{lem}\label{lemma-S=SR} We have
$S(\alpha, n, k) =   S(\alpha, n, n) \cdot R(\alpha, n, n-k)$.
\end{lem}
\begin{proof}
This is a straightforward calculation - just write
\begin{align} S(\alpha, n, k) &\ = \  \left( \prod_{i=1}^{n} a_i(\alpha) \right) \left( \dfrac{ \displaystyle\sum_{1\leq i_1 < \cdots < i_{n-k} \leq n} 1/(a_{i_1}(\alpha) \cdots a_{i_{n-k}}(\alpha))}{\dbinom{n}{n-k}}\right) \nonumber\\
&\ = \  S(\alpha, n, n) \cdot R(\alpha, n, n-k).
\end{align}
\end{proof}

We can now prove a strengthening of Proposition \ref{bound-limsup-babyversion}.

\begin{thm}\label{bound-limsup-improvedversion}
For almost all $\alpha$, and any $c \in (0,1)$, we have
\be K_0 \ \le \  \limsup_{n \to \infty} S(\alpha, n, cn)^{1/cn} \ \le \  K_0^{1/c}(K_{-1})^{1-\frac{1}{c}}.\ee
\end{thm}

\begin{proof}
We know by Lemma \ref{lemma-S=SR} and Maclaurin's inequalities \eqref{Maclaurin-inequalities} applied to the positive sequence $X=(1/a_i)_{i\geq 1}$ that \bea S(\alpha, n, cn)^{1/cn} &\ = \ & S(\alpha, n, n)^{1/cn} R(\alpha, n, (1-c)n)^{1/cn} \nonumber\\
&\ = \ & \left(S(\alpha, n, n)^{1/n}\right)^{1/c} \left(R(\alpha, n, (1-c)n)^{1/(1-c)n}\right)^{(1-c)/c} \nonumber\\ &\le & \left(S(\alpha, n, n)^{1/n}\right)^{1/c} \left(R(\alpha, n,1)\right)^{(1-c)/c}. \eea
Taking limits and using \eqref{harmonic-mean2}, we get the claim.
\end{proof}

Note that the limiting behavior of $S(\alpha, n, k(n))^{1/k(n)}$ does not depend on the values of the first $M$ continued fraction digits of $\alpha$, for any finite number $M$. Suppose that $a_i(\alpha')$ and $a_i(\alpha)$ agree for all $i > M$. Then
\begin{align}
\lim_{n\to\infty}S(\alpha,n,k(n))^{1/k(n)}\ =\ L\ \ \iff\ \ \lim_{n\to\infty}S(\alpha',n,k(n))^{1/k(n)}\ = \ L\label{change-M-entries}
\end{align}
where $L$ can be finite of infinite.
In fact, if $k(n)=o(n)$ as $n\to\infty$ then number of terms in $S(\alpha,n,k(n))$ not involving the digits $a_1(\alpha),\ldots, a_M(\alpha)$ is ${n-M \choose k(n)}$, which is very close to ${n \choose k(n)}$, namely ${n-M\choose k(n)}/{n\choose k(n)}=1-M k(n)/n+O((k(n)/n)^2)$. Therefore the contribution of terms involving $a_1(\alpha),\ldots,a_M(\alpha)$ is negligible.
If $k(n)=c n$, asymptotically the ratio between the number of terms not involving the first $M$ digits and ${n\choose cn}$ is $(1-c)^M$, but each term consists of a product of $\lceil cn\rceil$  continued fraction digits, of which at most $M$ come from the set $\{a_1(\alpha),\ldots,a_M(\alpha)\}$, and therefore their contribution to the limit is irrelevant.

Another way of of seeing that the $\limsup$-version of \eqref{change-M-entries} holds for fixed $k$ is the following: since $S(\alpha, n, k)^{1/k}$ is monotonic increasing in the $a_i$, and all the $a_i$ are positive, we can find a number $C$ such that $Ca_i(\alpha) > a_i(\alpha')$ and $Ca_i(\alpha') > a_i(\alpha)$ for all $i$. By inspection the means are linear with respect to multiplication of the vector $(a_1, a_2, \dots)$ by a constant $C$. Thus, combining this with monotonicity we get that
\be \limsup_{n \to \infty}S(\alpha, n, k)^{1/k} \ = \  \infty \iff \limsup_{n \to \infty}S(\alpha', n, k)^{1/k} \ = \  \infty.\nonumber\ee
A consequence of this fact is that if $X = (x_1, x_2, \dots)= (f(1), f(2), \dots)$ where $f$ is any unbounded increasing function, then $\lim_{n \to \infty} S(X,n, k)^{1/k} = \infty$ for any $k=k(n)$.

\begin{lem}\label{lem-Niculescu-applied-to-k=cn}
For any $\alpha \in \mathbb{R}$, any $c, d \in (0,1]$ and $t \in [0,1]$ such that $cn, dn, tcn, (1-t)dn$ are integers, we have
\be S(\alpha, n, t cn + (1-t)dn) \ \ge \  S(\alpha, n, cn)^{t} \cdot S(\alpha, n, dn)^{1-t}.\ee
\end{lem}
\begin{proof}
This is a direct application of Proposition \ref{prop-Niculescu}.
\end{proof}

It is natural to investigate the limit $\lim_{n\to\infty}S(\alpha,n,cn)^{1/cn}$ as a function of $c$. However, since we have not proved that for almost every $\alpha$ this limit exists, we will have to assume that it does.
Define
\begin{align}
 F_{+}^{\alpha}(c) &\ = \  F_{+}(c) \ := \  \limsup_{n \to \infty} S(\alpha, n, cn)^{1/cn}, \nonumber\\
 F_{-}^{\alpha}(c) &\ = \  F_{-}(c) \ := \  \liminf_{n \to \infty} S(\alpha, n, cn)^{1/cn}.\end{align}

\begin{conj}\label{hyp1} For almost all $\alpha$ and all $0<c\leq 1$, we have
$F_{+}(c) = F_{-}(c)$.
In this case we write $F(c):= \lim_{n \to \infty} S(\alpha, n, cn)^{1/cn}$.
\end{conj}

We investigated the plausibility of Conjecture \ref{hyp1} by looking at the averages $S(\alpha,n,cn)^{1/cn}$ for various values of $\alpha$ (such as $\pi-3$, Euler-Mascheroni constant $\gamma$, and $\sin(1)$) that are believed to be \emph{typical} (the averages $S(\alpha,n,n)^{1/n}$ are believed to converge to $K_0$  as $n\to\infty$ for such $\alpha$'s), and $0<c\leq 1$.

Figure \ref{fig-evidence-conj1-0} shows the function $c=\frac{k}{n}\mapsto S(\alpha,n,k)^{1/k}$ for $\alpha= \pi-3,\gamma,\sin(1)$ and various values of $n$. Figure \ref{fig-evidence-conj1}  specifically looks at the convergence of $S(\alpha,n,cn)^{1/cn}$ for $\alpha$ as above and specific values of $c$. It is reasonable to believe that $\lim_{n\to\infty}S(\alpha,n,cn)^{1/cn}$ exists for these  $\alpha$'s, and the limit is the same as for typical $\alpha$. To compute the averages $S(\alpha,n,k)^{1/k}$ we use the following identity for elementary symmetric polynomials: if
\be E(n,k)[x_1,\ldots,x_n]\ = \ \sum_{1\leq i_1\ < \ \cdots\ < \ i_k\leq n}x_{i_1}\cdots x_{i_k},\ee then
\be E(n,k)[x_1,\ldots,x_n]\ = \ x_n E(n-1,k-1)[x_1,\ldots,x_{n-1}]+E(n-1,k)[x_1,\ldots,x_{n-1}].\ee

\begin{center}
\begin{figure}
\includegraphics[width=14cm]{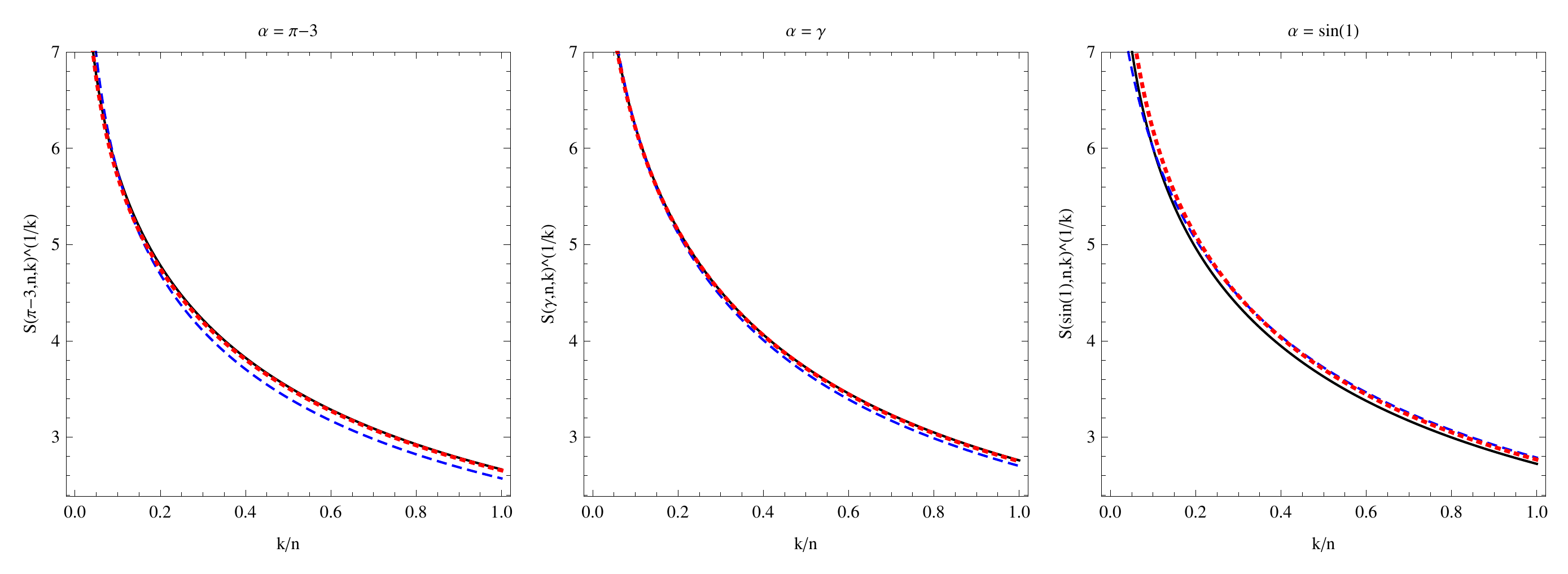}
\caption{Evidence for Conjecture \ref{hyp1}. Plot of $\frac{k}{n}\mapsto S(\alpha,n,k)^{1/k}$  for $\alpha=\pi-3,\gamma,\sin(1)$ and $n=600$  (dashed blue), $800$ (dotted red), $1000$ (solid black).}\label{fig-evidence-conj1-0}
\vspace{.3cm}
\includegraphics[width=14cm]{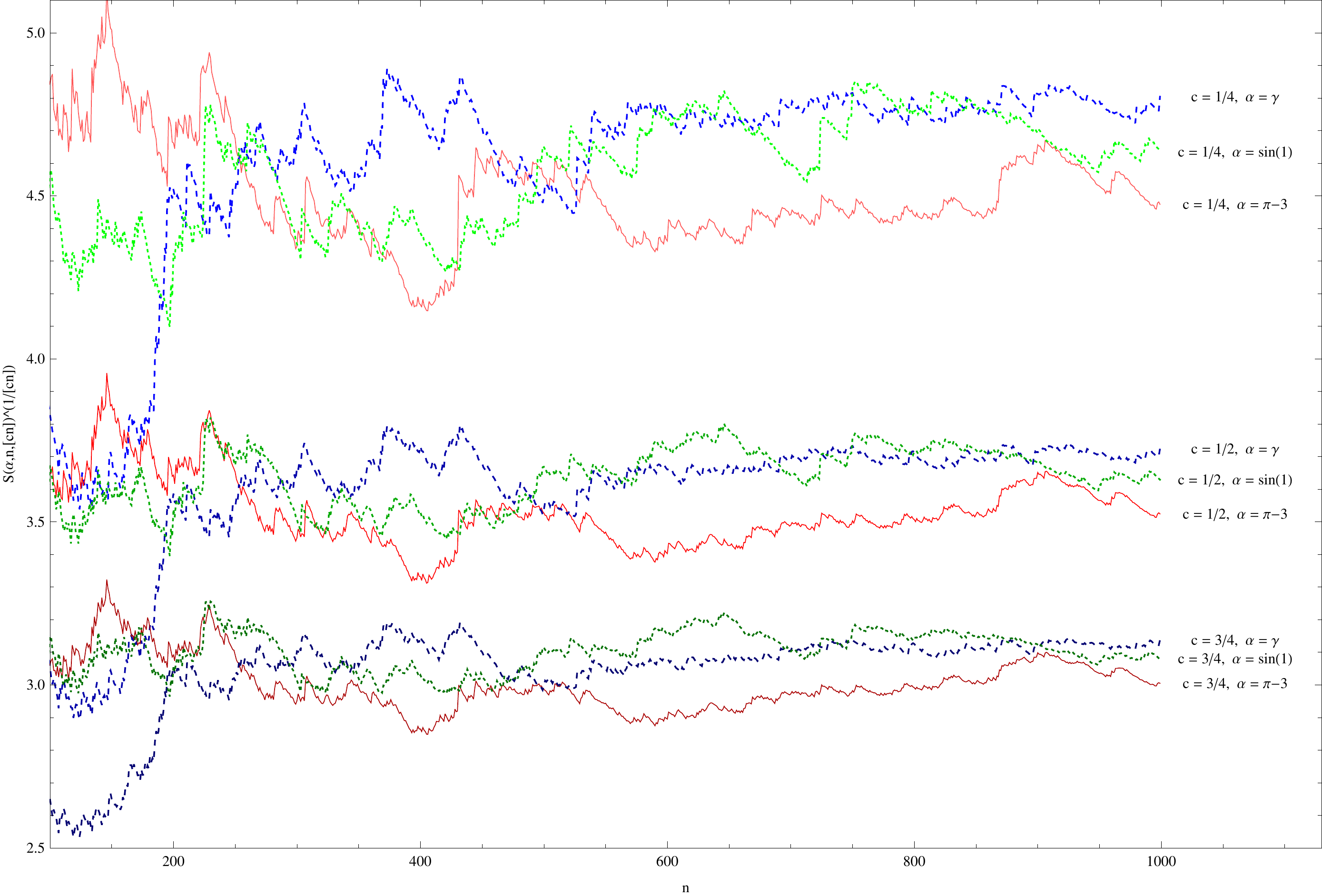}
\caption{Evidence for Conjecture \ref{hyp1}. Plot of $n\mapsto S(\alpha,n,cn)^{1/cn}$  for $c=1/4$ (top), $1/2$ (middle), $3/4$ (bottom) and $\alpha=\pi-3$ (solid red), $\gamma$ (dashed blue), $\sin(1)$ (dotted green).}\label{fig-evidence-conj1}
\end{figure}
\end{center}

\begin{prop}\label{F-is-continuous}
Assume Conjecture \ref{hyp1}. Then the function $c\mapsto F(c)$ is continuous on $(0,1]$.
\end{prop}
\begin{proof}
Assuming Conjecture \ref{hyp1}, it follows from 
Lemma \ref{lem-Niculescu-applied-to-k=cn} that
\be \log{F(tc+(1-t)d)} \ \ge \  \left(\frac{1}{tc+(1-t)d}\right)\left(tc \log{F(c)} +(1-t)d \log{F(d)}\right). \label{log-concavity-like}\ee
By fixing $d > c$ and letting $t \to 1$, we get
\be \lim_{x \to c^+} \log{F(x)} \ \ge \  \log{F(c)}; \ee
however, as  $\log{F(c)}$ is non-increasing by Maclaurin's inequalities \eqref{Maclaurin-inequalities} we must have equality. Similarly,  for small $\epsilon > 0$, we get 
\bea \log F\left(c+(1-2t)\epsilon\right) &\ = \ & \log F\left(t(c-\epsilon)+(1-t)(c+\epsilon)\right) \nonumber\\
&\ge& \left(\frac{1}{c+(1-2t)\epsilon}\right)\big(t(c-\epsilon) \log F\left(c-\epsilon\right) \nonumber\\  & & \ \ \ +\ (1-t)(c+\epsilon) \log F\left(c+\epsilon\right)\big).\eea

Setting $t=1/2$ yields
\bea \log F(c) & \ \ge \ & \left(\frac{1}{c+\epsilon+c-\epsilon}\right) (c-\epsilon) \log F(c-\epsilon)   + (c+\epsilon) \log F(c+\epsilon), \eea
then taking the limit as $\epsilon \to 0$ gives
\be \log F(c) \ \ge \  \frac{1}{2} \lim_{x \to c^-} \log F(x)  + \frac{1}{2} \lim_{x \to c^+} \log F(x) \ =\ \frac{1}{2}\lim_{x \to c^-} \log F(x)  + \frac{1}{2} \log F(c).\ee
Combining this with the monotonicity of $F$ shows that $\log{F}$ is continuous, and exponentiation proves the proposition.
\end{proof}

\begin{prop}
Assume Conjecture \ref{hyp1}. Then the function $R : [0,1] \to [1/K, 1/K_{-1}]$ defined by
\be R(c)\ =\
  \begin{cases}
      \hfill \lim_{n \to \infty} R(\alpha, n, cn)^{1/cn}   \hfill & \text{ {\rm if} $c > 0$} \\
      \hfill 1/K_{-1} \hfill & \text{ {\rm if} $c=0$} \\
  \end{cases} \ee
is uniformly continuous.
\end{prop}
\begin{proof}
This follows from  Lemma \ref{lemma-S=SR}, Proposition \ref{F-is-continuous} and the Heine-Cantor theorem.
\end{proof}

\begin{lem}\label{lem-n-choose-cn}
For any constant $0 < c < 1$, we have
\be \lim_{n \to \infty} \dbinom{n}{\lceil cn\rceil}^{1/\lceil cn\rceil}\ =\ \dfrac{(1-c)^{1-\frac{1}{c}}}{c} \ee
\end{lem}
\begin{proof}
Taking the logarithm, we get
\begin{eqnarray}
 \lim_{n \to \infty} \log \dbinom{n}{\lceil cn\rceil}^{1/\lceil cn\rceil}&\ = \ & \lim_{n \to \infty} \dfrac{\log{\frac{n!}{\lceil cn\rceil !\lceil(1-c)n\rceil !}}}{\lceil cn\rceil } \nonumber\\
 &\ =\ & \lim_{n \to \infty} \dfrac{\log{n!}-\log{\lceil cn\rceil!}-\log{\lceil(1-c)n\rceil!}}{\lceil cn\rceil}.
 \end{eqnarray}
Using Stirling's formula gives
\bea \lim_{n \to \infty} \log \dbinom{n}{\lceil cn\rceil}^{1/\lceil cn\rceil}  & \ = \ & \lim_{n \to \infty} \dfrac{n\log{n}-nc\log{(cn)}-(1-c)n\log{((1-c)n)}+O(\log{n})}{cn} \nonumber \\
& \ = \ & -\log{c}+\frac{(c-1)}{c}\log{(1-c)}. \eea
Exponentiation gives the desired result.
\end{proof}

\begin{lem}\label{difference-in-terms-S(alpha,n,cn)^1/cn}
For any $c \in (0,1]$ and almost all $\alpha$ the difference between consecutive terms in the sequence $\{S(\alpha, n, cn)^{1/cn}\}_{n \in \mathbb{N}}$ goes to zero. Moreover, the difference between the $n$\textsuperscript{th} and the $(n+1)$\textsuperscript{st} terms is $O\left(\frac{\log{n}}{n}\right)$.
\end{lem}

\begin{proof}
We have two cases to consider: when $\lceil c(n+1)\rceil = \lceil cn\rceil$ and when $\lceil c(n+1)\rceil = \lceil cn\rceil+1$. Let $k=\lceil cn\rceil$ and $x_i = a_i(\alpha)$. In the first case, the difference between the $n$\textsuperscript{th} and the $(n+1)$\textsuperscript{st} terms is
\be \left|S(\alpha, n, k)^{1/k}\left(1-\left(\frac{S(\alpha, n+1, k)}{S(\alpha, n, k)}\right)^{1/k}\right)\right|\ee
which, for sufficiently large $n$, can be bounded above by
\begin{align} &K^{1/c}\left(\left(\frac{\sum_{i_1 < \cdots < i_k}^{n+1}x_{i_1}\cdots x_{i_k}}{\sum_{i_1 < \cdots < i_k}^{n}x_{i_1}\cdots x_{i_k}}\right)^{1/k}-1\right)\\ &\ = \  K^{1/c}\left( \left(1+ x_{n+1}\dfrac{\sum_{i_1 < \cdots < i_{k-1}}^{n}x_{i_1}\cdots x_{i_{k-1}}}{\sum_{i_1 < \cdots < i_k}^{n}x_{i_1}\cdots x_{i_k}}\right)^{1/k} - 1   \right). \end{align}
As all the $x_i \geq 1$, the fraction multiplying $x_{n+1}$ is $\leq 1$. For almost all $\alpha$ and for large enough $n$,  we have $x_{n+1} \ < \  n^2$, and this difference is no bigger than
\be K^{1/c}((1+n^2)^{1/n})^{1/c} - 1) \ = \  O\left(\frac{\log{n}}{n}\right).\ee
Next we consider the case when $\lceil c(n+1)\rceil = \lceil cn\rceil+1$. The difference is now
\bea & & \left|S(\alpha, n, k)^{1/k}\left(1-\left(\frac{S(\alpha, n+1, k+1)}{S(\alpha, n, k)}\right)^{1/(k+1)}S(\alpha, n, k)^{-1/(k^2+k)}\right)\right| \nonumber\\ & & \ \ \ \ \le \  K^{1/c}\left(\left(\frac{\sum_{i_1 < \cdots < i_{k+1}}^{n+1}x_{i_1}\cdots x_{i_{k+1}}}{\sum_{i_1 < \cdots < i_k}^{n}x_{i_1}\cdots x_{i_k}}\right)^{1/(k+1)}(1+ O(1/n))-1\right).\eea
As
\begin{align} 1 &\ \le \ \frac{\sum_{i_1 < \cdots < i_{k+1}}^{n+1}x_{i_1}\cdots x_{i_{k+1}}}{\sum_{i_1 < \cdots < i_k}^{n}x_{i_1}\cdots x_{i_k}} \ = \ x_{n+1}+ \frac{\sum_{i_1 < \cdots < i_{k+1}}^{n}x_{i_1}\cdots x_{i_{k+1}}}{\sum_{i_1 < \cdots < i_k}^{n}x_{i_1}\cdots x_{i_k}}\nonumber \\
&\ < \  x_{n+1}+n\cdot\max_{i \leq n}{x_i}, \end{align}
which is less than $n^3$ for large enough $n$ and for almost all $\alpha$, we find 
\be \left(\frac{\sum_{i_1 < \cdots < i_{k+1}}^{n+1}x_{i_1}\cdots x_{i_{k+1}}}{\sum_{i_1 < \cdots < i_k}^{n}x_{i_1}\cdots x_{i_k}}\right)^{1/(k+1)} \ = \  1 + O\left(\frac{\log{n}}{n}\right). \ee \\
Thus the claim holds in both cases.
\end{proof}

The following proposition is a corollary of Lemma \ref{difference-in-terms-S(alpha,n,cn)^1/cn}.

\begin{prop}\label{if-not-converge-then}
For almost all $\alpha$, if the sequence $\{S(\alpha, n, cn)^{1/cn})\}_{n \in \mathbb{N}}$ does not converge to a limit then its set of limit points is a non-empty interval inside $[K, K^{1/c}]$.
\end{prop}

\begin{proof}
Since the sequence must lie in this compact interval eventually, it must have a limit point $x$. If the sequence does not converge to this limit, there must be a second limit point $y$ with, say, $y-x = \epsilon > 0$. If there are no limit points between $x$ and $y$, then infinitely often consecutive terms of the sequence must differ by at least $\epsilon/3$. This cannot happen for almost all $\alpha$ by the Lemma \ref{difference-in-terms-S(alpha,n,cn)^1/cn}, and so the set of limit points cannot have any gaps between its supremum and infimum. Since the set of limit points is closed, it must be a closed interval.
\end{proof}

\begin{lem}
Let $f(n)$ be some integer-valued function such that $f(n) > n$ for all $n$, and let $x_i=a_i(\alpha)$. Then for almost all $\alpha$ we have

\be  \lim_{n \to \infty} \dfrac{\left(x_{n+1} \cdots x_{f(n)}\right)^{1/f(n)}}{ K_0^{\frac{f(n)-n}{f(n)}} } \ = \  1. \ee
\end{lem}
\begin{proof}
This follows from the fact that the sequence of geometric means is (almost always) Cauchy with limit $K_0$. More explicitly,

\begin{align} &\left(x_{1} \cdots x_{n}\right)^{1/n} - \left(x_{1} \cdots x_{f(n)}\right)^{1/f(n)}\nonumber\\
&\ = \  \left(x_{1} \cdots x_{n}\right)^{1/n} \left(1-\left(x_{1} \cdots x_{n}\right)^{1/f(n) - 1/n}\left(x_{n+1} \cdots x_{f(n)}\right)^{1/f(n)}\right). \end{align}
This quantity must go to zero as $n \to \infty$, which implies that the limit in question is 1.
\end{proof}

\begin{conj}\label{hyp2} There exist constants $a,b \in \mathbb{R}^+$  such that for almost all $\alpha$ and each $c \in (0,1]$,
\be \lim_{n \to \infty} S(\alpha, n, cn)^{1/cn} \ = \  \frac{K_0}{b}\big(b^{1/c^{a}} \big). \ee
\end{conj}
Observe that such functions obey the log concavity-like inequality \eqref{log-concavity-like}, and qualitatively agree with the functions in Figure \ref{fig-evidence-conj1} (top).

Notice that if Conjecture \ref{hyp2} is correct, then for almost every $\alpha$ we have $F(c)$ grows without bound as $c \to 0^+$. Then we can replace the assumption $k(n)=o(\log\log n)$ in Theorem \ref{thm1} by $k(n) = o(n)$. We obtain that for almost every $\alpha$
\be \lim_{n \to \infty} S(\alpha, n, k)^{1/k} \ = \  \infty,\ee
completing our characterization on each side of the phase transition. In Theorem \ref{thm2.7} we obtain the same result assuming Conjecture \ref{hyp1} (which is weaker than Conjecture \ref{hyp2}) and the unrelated Conjecture \ref{conj-periodic-c} (see Section \ref{section-backtotypicalalpha}).

\section{Averages for quadratic irrational $\alpha$}\label{section-periodic}

Lagrange's theorem (see e.g. \cite{Miller-TaklooBighash-book}) states that $\alpha$ has a (pre)periodic continued fraction expansion if and only if it is a quadratic surd. These real numbers in general do not have the same asymptotic means as typical $\alpha$.
Let us restrict our attention to periodic $\alpha=[a_1,a_2,\ldots,a_L,a_1,a_2\ldots, a_L,\ldots]$, the preperiodic case being similar, see \eqref{change-M-entries}. In this case the value of the arithmetic and geometric means are independent of the number of periods we include, as long as it is integral. This does not extend to the other elementary symmetric means.

Let us consider an arbitrary sequence of positive real numbers (not necessarily integers) with period $L$, $X=(x_1,\ldots,x_L,x_1,\ldots)$. We want to study the function
\be (k,c)\mapsto F_X(k,c)\ :=\ S(X,kL,\lceil ckL\rceil )^{1/\lceil ckL\rceil}\label{F_X(k,c)}\ee
for $k\geq 1$. Notice that, for fixed $k$, the function $c\mapsto F_X(k,c)$ is non-increasing by MacLaurin's inequalities \eqref{Maclaurin-inequalities} and piecewise constant. In particular, for $c\in(0,\frac{1}{kL}]$, $F_X(k,c)=S(X,kL,1)^{1/1}=S(X,L,1)=(x_1+\cdots+x_L)/L$. It is therefore natural to define, for every $k$, $F_X(k,0):=(x_1+\cdots+x_L)/L$ and consider each $F_X(k,c)$ as a function on $0\leq c\leq 1$. 

We investigate the case of $2$-periodic sequences $X=(x,y,x,y,\ldots)$ first. We have
\be
F_X(k,c)\ = \ S(X,2k,\lceil 2ck\rceil)^{1/\lceil 2ck\rceil}\ = \ \frac{1}{{2k\choose \lceil 2ck \rceil}}\sum_{j=0}^{\lceil 2ck \rceil} {k\choose j}{k\choose \lceil 2ck \rceil-j}x^j y^{\lceil 2ck \rceil-j},\label{F_X(k,c)forL=2}
\ee
see Figure \ref{threeFXkc}.
\begin{center}
\begin{figure}[h]
\includegraphics[width=16cm]{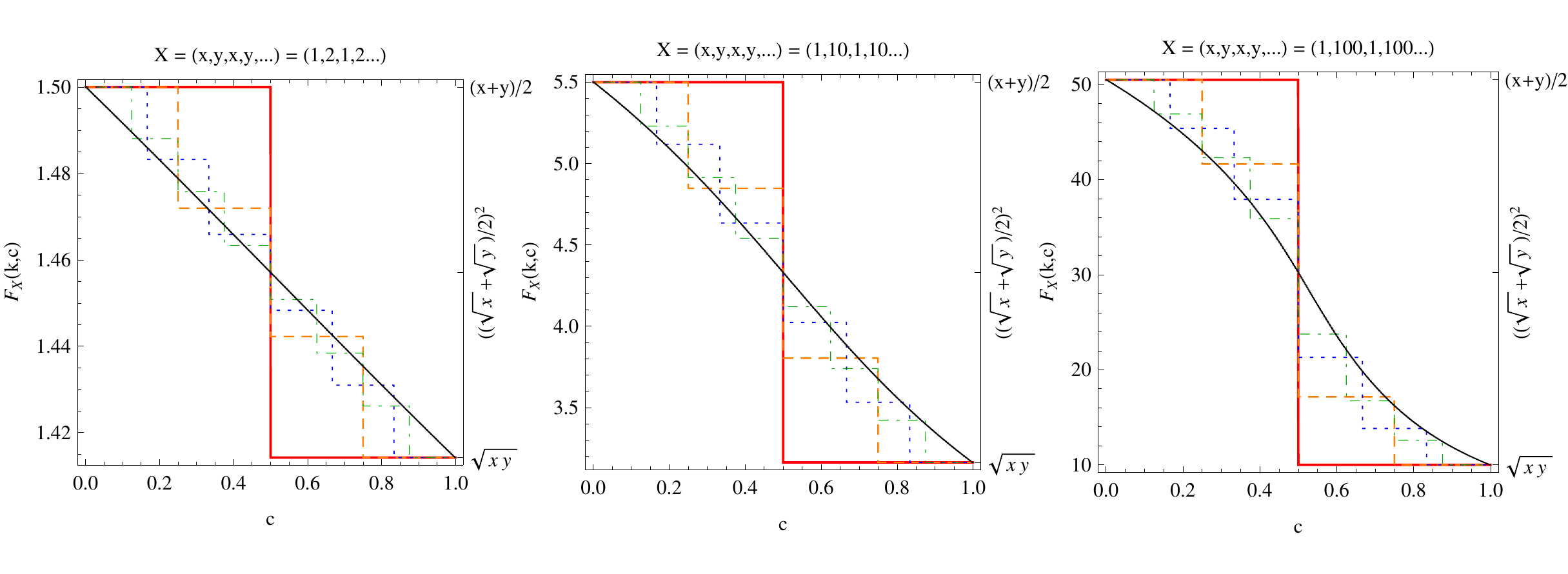}
\caption{The function $c\mapsto F_X(k,c)$ for three different $X$ of period $L=2$ and $k=1$ (solid red), $k=2$ (dashed orange), $k=3$ (dotted blue), $k=4$ (dash-dotted green), $k=200$ (solid black).}\label{threeFXkc} 
\end{figure}
\end{center}
The following lemma addresses the convergence as $k\to\infty$ for the sequence \eqref{F_X(k,c)forL=2} at $c=1/2$, where $F_X(k,1/2)=S(X,2k,k)^{1/k}$. Monotonicity in $k$ and an explicit formula for the limit in terms of $x$ and $y$.

\begin{lem}\label{lemma-monotonicity}
 Let  $X = (x,y,x,y, \dots)$ be a 2-periodic sequence of positive real numbers. Then for sufficiently large $k \in \mathbb{N}$, we have
\be S(X, 2k, k)^{1/k} \ \ge \  S(X, 2k+2, k+1)^{1/(k+1)}.\label{statement-lemma-monotonicity} \ee
Moreover
\be\label{formula-limit}
\lim_{k\to\infty}S(X,2k,k)^{1/k}\ = \ \left(\frac{x^{1/2}+y^{1/2}}{2}\right)^2, \ee
which is the $\frac{1}{2}$-H\"older mean of $x$ and $y$.
\end{lem}

\begin{proof}
If $x=y$ then the lemma is trivially true and \eqref{statement-lemma-monotonicity} is actually an equality.  Thus we can assume that $x\neq y$. We want to show that $S(X,2k,k)^{\frac{1}{k}}> S(X,2k+2,k+1)^{\frac{1}{k+1}}$.
We can write
\be S(X,2k,k)\ = \ \frac{1}{{2k\choose k}}\sum_{j=0}^k{k\choose j}^2x^j y^{k-j}\ = \ \frac{y^k}{{2k\choose k}}\sum_{j=0}^k{k\choose j}^2 t^j\ee
with $t=x/y$. Without loss of generality we can assume that $0<t<1$.
Recall the Legendre polynomials $P_k(u)$, defined by the recursive formula
\begin{equation}
(k+1)P_{k+1}(u)\ = \ (2k+1)u P_k(u)-k P_{k-1}(u),\:\:k\geq 2\label{recursion-Pk}
\end{equation}
with $P_0(u)=1$ and $P_1(u)=u$. An explicit formula for $P_k(u)$ is
\be P_k(u)\ = \ \frac{1}{2^k}\sum_{j\ = \ 0}^k{k\choose j}^2(u-1)^{k-j}(u+1)^j.\ee
This allows us to write
\be \sum_{j=0}^k{k\choose j}^2 t^j\ = \ (1-t)^k P_k(\tfrac{1+t}{1-t})\ee
and
\begin{align}
S(X,2k,k)^{\frac{1}{k}}& \ > \  S(X,2k+2,k+1)^{\frac{1}{k+1}}\nonumber\\
\iff\left(\frac{y^k\sum_{j=0}^k{k\choose j}^2 t^j}{{2k\choose k}}\right)^{\frac{1}{k}} & \ > \ \left(\frac{y^{k+1}\sum_{j=0}^{k+1}{k+1\choose j}^2 t^j}{{2k+2\choose k+1}}\right)^{\frac{1}{k+1}}\nonumber\\
\iff\left(\frac{P_k(u)}{{2k\choose k}}\right)^{\frac{1}{k}} & \ > \ \left(\frac{P_{k+1}(u)}{{2k+2\choose k+1}}\right)^{\frac{1}{k+1}},\label{wtp}
\end{align}
where $u=\frac{1+t}{1-t}>1$.
We show that \eqref{wtp} holds for sufficiently large $k$.

For $u=1$ we have $P_k(1)=1$. Using Stirling's formula, one can check that
\begin{align}
{2k \choose k}^{-\frac{1}{k}}\ = \ \frac{1}{4}+\frac{\log k+\log\pi}{8k}+O\!\left(k^{-\frac{3}{2}}\right),\label{asymptotic (2k choose k)^-1/k}
\end{align}
and, since the function $k\mapsto\frac{\log k+\log \pi}{k}$ is strictly decreasing for $k\geq1$, the inequality \eqref{wtp} holds when $u=1$ for sufficiently large $k$. The expansion (for fixed $k$) at $u\sim 1$ is
\be P_k(u)\ = \ 1+\frac{k(k+1)}{2}(u-1)+O((u-1)^2)\ee
(see 22.5.37 and 22.2.3 in \cite{Abramowitz-Stegun}), and
\begin{align}
\left.\frac{\mathrm d}{\mathrm d u}\left( \frac{P_{k}(u)}{{2k\choose k}}\right)^{\frac{1}{k}}\right|_{u=1}\ = \ \frac{k+1}{2}{2k \choose k}^{-\frac{1}{k}}>0\nonumber
\end{align}
by \eqref{asymptotic (2k choose k)^-1/k} for sufficiently large $k$.
Therefore, by continuity of $u\mapsto P_k(u)$, \eqref{wtp} is true in some  neighborhood of $u=1$, i.e., there exists $\delta>0$  such that \eqref{wtp} holds for  $u\in(1,1+\delta]$ and all sufficiently large $k$.

To consider the case of arbitrary $u\geq 1+\delta$ we use the following \emph{generalized Laplace-Heine asymptotic formula} (see 8.21.3 in \cite{Szego}) for $P_k(u)$. Let $z=u+\sqrt{u^2-1}$. We have $z>1$ and for any $p\geq 1$
\begin{align}
 P_k(u)\ = \ \frac{(2k-1)!!}{(2k)!!} z^k \sum_{l=0}^{p-1} \frac{((2l-1)!!)^2(2k-2l-1)!!}{(2l)!!(2k-1)!!}z^{-2l}(1-z^{-2})^{-l-\frac{1}{2}}+O(k^{-p-\frac{1}{2}}z^k),\label{generalized-Laplace-Heine-formula}
\end{align}
where the big-$O$ constant is uniform for arbitrary $u\geq 1+\delta$.
Notice that all terms in \eqref{generalized-Laplace-Heine-formula} are strictly positive. Observe that $z^{-\frac{1}{2}}(1-z^{-2})^{-\frac{1}{2}}=\frac{1}{\sqrt{2}}(u^2-1)^{-\frac{1}{4}}$, and that
\begin{align}\frac{(2k-1)!!}{(2k)!!}\ = \ \frac{\frac{(2k-1)!}{2^{k-1}(k-1)!}}{2^k k!}\ = \ \frac{\Gamma(k+\tfrac{1}{2})}{\sqrt{\pi}\Gamma(k+1)}.
\end{align}
For $p=2$, \eqref{generalized-Laplace-Heine-formula} yields
\begin{align}
  P_k(u)\ = \ &\frac{\Gamma(k+\tfrac{1}{2})}{\sqrt{2 \pi}\Gamma(k+1)}\frac{z
^{k+\frac{1}{2}}}{(u^2-1)^{\frac{1}{4}}}
\left(1+
\frac{\Gamma(k-\tfrac{1}{2})}{4\Gamma(k+\tfrac{1}{2})}z^{-2}(1-z^{-2})^{-1}\right)+O(k^{-\frac{5}{2}}z^k).\label{generalized-Laplace-Heine-formula-p=2}
\end{align}
Now we use the following asymptotic formulas (as $k\to\infty$)
\begin{align}
 \frac{\sqrt{k}\Gamma(k+\tfrac{1}{2})}{\Gamma(k+1)}&\ = \ 1-\frac{1}{8k}+O(k^{-2})\nonumber\\
\left(\frac{c_1}{k}\right)^{\frac{1}{2k}}&\ = \ 1-\frac{\log k-\log c_1}{2k}+O(k^{-2})\nonumber\\
\frac{\Gamma(k-\tfrac{1}{2})}{\Gamma(k+\tfrac{1}{2})}&\ = \ \frac{1}{k}+O(k^{-\frac{3}{2}})\nonumber
\end{align}
in \eqref{generalized-Laplace-Heine-formula-p=2}.
We obtain, for sufficiently large $k$,
\begin{align}
   P_k(u)\ = \ & z^k\left(1-\frac{1}{8k}+O(k^{-2})\right)\left(1-\frac{\log k-\log c_1}{2k}+O(k^{-2})\right)\nonumber\\
& \cdot\left(1+\frac{c_2}{k}+O(k^{-\frac{3}{2}})\right)\left(1+O(k^{-\frac{5}{2}})\right)\nonumber\\
\ = \ & z^k\left(1-\frac{\log k+\frac{1}{4}-\log c_1-2c_2}{2k}+O(k^{-3/2})\right)\nonumber
\end{align}
where
 $c_1=\frac{z}{2\pi\sqrt{u^2-1}}$, $c_2=\frac{z^{-2}(1-z^2)^{-1}}{4}$, and the constants implied by the $O$-notations depend only on $u$.
This implies
\begin{align}
(P_k(u))^{\frac{1}{k}}\ = \ & z\left(1-\frac{\log k+\frac{1}{4}-\log c_1-c_2}{2k^2}+O(k^{-5/2})\right)\nonumber
\end{align}
and, by \eqref{asymptotic (2k choose k)^-1/k},
\begin{align}
\left(\frac{P_k(u)}{{2k\choose k}}\right)^{\frac{1}{k}}&\ = \  z\left(1-\frac{\log k+\frac{1}{4}-\log c_1-c_2}{2k^2}+O(k^{-5/2})\right)\!\left(\frac{1}{4}+\frac{\log k+\log\pi}{8k}+O(k^{-\frac{3}{2}})\right)\nonumber\\
&\ = \ \frac{z}{4}\left(1+\frac{\log k+\log \pi}{k}+O(k^{-3/2})\right).\label{asymptotic-exp-P_k(u)/binomial^1/k}
\end{align}
As before, the monotonicity of $k\mapsto \frac{\log k+\log \pi}{k}$ gives \eqref{wtp} for arbitrary $u\geq1+\delta$ for sufficiently large $k$. This concludes the proof of \eqref{statement-lemma-monotonicity}. Now \eqref{formula-limit} follows from \eqref{asymptotic-exp-P_k(u)/binomial^1/k} since
\be S(X,2k,k)^{1/k}\ = \ y (1-t)\left(\frac{P_k(u)}{{2k\choose k}}\right)^{1/k}\to y(1-t)\frac{u+\sqrt{u^2-1}}{4}\ = \ y\left(\frac{1+\sqrt t}{2}\right)^2.\label{final-formula-2k-k}\ee
\end{proof}

For the example of  $\alpha=\sqrt3-1=[1,2,1,2,1,2,\ldots]$ mentioned in the introduction we get
$\lim_{n\to\infty}F_X(k,1/2)=\lim_{n\to\infty} S(\alpha,2n,n)^{1/n}=
\frac{3+2\sqrt2}{4}$, see also Figure \ref{threeFXkc} (left).

For any fixed $2$-periodic $X$ we just showed in Lemma \ref{lemma-monotonicity} that for $c=1/2$, the sequence $\{F_X(k,1/2)\}_{k\geq 1}$ is monotonic (and convergent). It would be naturally to conjecture that this sequence is monotonic for every $c$. This, however, is not true, as it can be seen already in Figure \ref{threeFXkc}. For instance, at $c=1/3$ we see that $F_X(1,1/3)<F_X(3,1/3)<F_X(4,1/3)<F_X(2,1/3)$. Figure \ref{nonmonotonicity} addresses the question of monotonicity in $k$ for various values of $c$ more directly: it is clear that the sequence $\{F_X(k,c)\}_{k\geq 1}$ is monotonic only at $c=1/2$.
The same figure  also suggests that, for fixed $X$ and $0\leq c\leq 1$, the sequence $\{F_X(k,c)\}_{k\geq 1}$ converges to a limit, notwithstanding the lack of monotonicity.
\begin{center}
\begin{figure}[h]
\includegraphics[width=14.8cm]{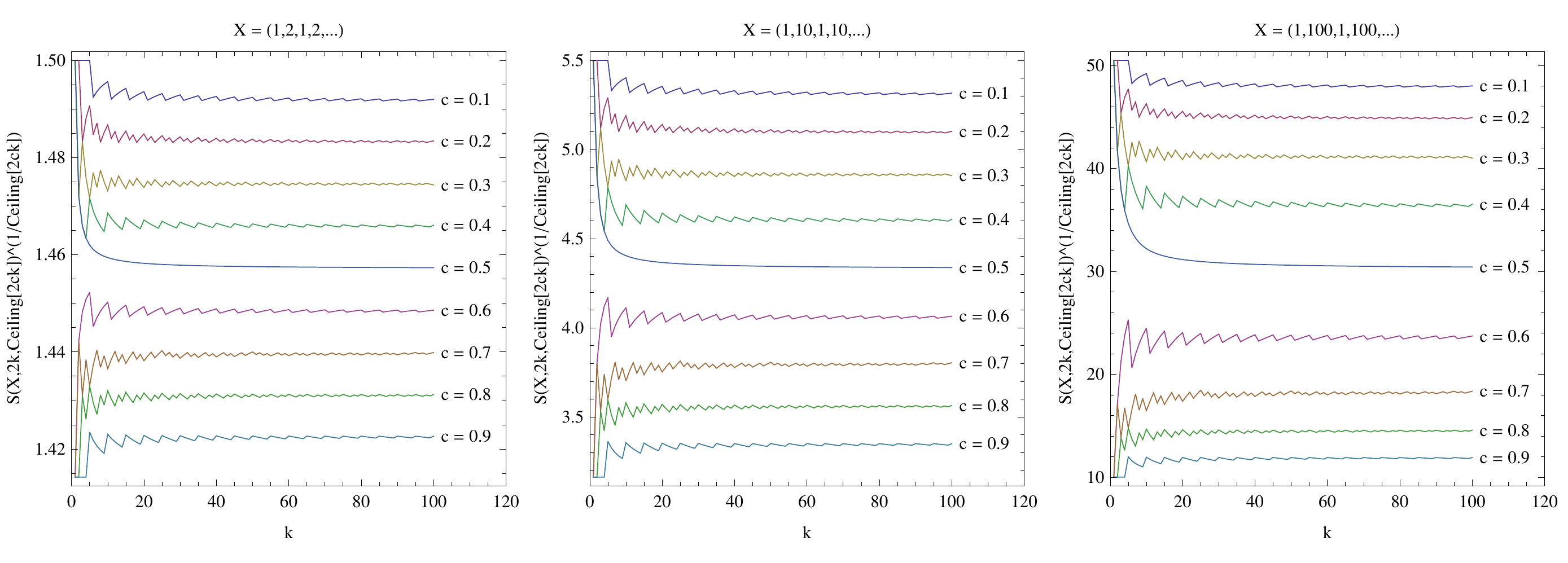}
\caption{Plot of the function $k\mapsto F_X(k,c)$ for three 2-periodic sequences $X$ and for $c\in\{.1,.2,\ldots,.9\}$. Notice that only for $c=1/2$ we have monotonicity in $k$ (Lemma \ref{lemma-monotonicity}).}\label{nonmonotonicity}
\end{figure}
\end{center}

Let us try to explore the above claim of convergence as $k\to\infty$ for $c\neq 1/2$. For simplicity, let us consider
the case of $c=1/3$. We want to prove the existence of the limit
$\lim_{k\to\infty}F_X(k,1/3)=\lim_{k\to\infty} S(X,2k,\lceil\frac{2}{3}k\rceil)^{1/\lceil \frac{2}{3}k\rceil}$ where $X=(x,y,x,y,x,y,\ldots)$.
The sequence $(2k,\lceil \frac{2}{3}k\rceil)$ consists of the following three subsequences: $(6k-2,2k)$, $(6k,2k)$, $(6k+2,2k+1)$. Let without loss of generality $0<t=x/y<1$. If we try to argue as in the proof of Lemma \ref{lemma-monotonicity}, we get that
\begin{eqnarray}
S(X,6k-2,2k)&\ = \ &\frac{1}{{6k-2\choose 2k}}\sum_{j=0}^{2k}{3k-1\choose j}{3k-1\choose 2k-j}x^j y^{2k-j}
\nonumber\\
&\ = \ & y^{2k}\frac{{3k-1\choose 2k}}{{6k-2\choose 2k}}\cdot {}_2F_1(-3k+1,-2k,k,t),\nonumber\\
S(X,6k,2k)&\ = \ & \frac{1}{{6k\choose 2k}}\sum_{j=0}^{2k}{3k\choose j}{3k\choose 2k-j}x^j y^{2k-j}\nonumber\\
&\ = \ & y^{2k}\frac{{3k\choose 2k}}{{6k\choose 2k}}\cdot {}_2F_1(-3k,-2k,1+k,t),\nonumber\\
S(X,6k+2,2k+1)&\ = \ & \frac{1}{{6k+2\choose 2k+1}}\sum_{j=0}^{2k+1}{3k+1\choose j}{3k+1\choose 2k+1-j}x^j y^{2k+1-j}\nonumber\\
&\ = \ & y^{2k}\frac{{3k+1\choose 2k+1}}{{6k+2\choose 2k+1}}\cdot {}_2F_1(-3k-1,-2k-1,1+k,t),\label{S(X,6k+m,3k+n)}
\end{eqnarray}
where ${}_2F_1$ is the hypergeometric function
\be{}_2F_1(a,b,c;z)\ = \ \sum_{n=0}^\infty\frac{(a)_n (b)_n}{(c_n)}\frac{z^n}{n!}\ee
and $(q)_n=\frac{\Gamma(q+1)}{\Gamma(q-n+1)}$ is the Pochhammer symbol\footnote{It is also possible to write the sums in \eqref{S(X,6k+m,3k+n)} in terms of Jacobi polynomials $P_n^{(\alpha,\beta)}(u)$ where $n,\alpha,\beta$ depend on $k$ and $u=\frac{1+t}{1-t}$ as in the proof of Lemma \ref{lemma-monotonicity}, see 22.5.44 in \cite{Abramowitz-Stegun}. This representation, however, does not seem to be useful for our purposes.}.
Let us notice the three limits
\be
\left(\frac{{3k-1\choose 2k}}{{6k-2\choose 2k}}\right)^{\frac{1}{2k}},\left(\frac{{3k\choose 2k}}{{6k\choose 2k}}\right)^{\frac{1}{2k}},\left(\frac{{3k+1\choose 2k+1}}{{6k+2\choose 2k+1}}\right)^{\frac{1}{2k+1}}\to\frac{2\sqrt3}{9}\label{three-limits}\ee as $k\to\infty$.
Numerically, we observe that each of the three functions
\begin{align}
&t\mapsto\left({}_2F_1(-3k+1,-2k,k,t)\right)^{\frac{1}{2k}}\nonumber\\
&t\mapsto\left({}_2F_1(-3k,-2k,1+k,t)\right)^{\frac{1}{2k}} \nonumber\\   
&t\mapsto\left({}_2F_1(-3k-1,-2k-1,1+k,t)\right)^{\frac{1}{2k+1}}\label{2F1C}
\end{align}
converges (monotonically) to a strictly increasing function of $t$, say $t\mapsto M(t)$, such that $M(0)=1$, $M'(0)=3$, $M(1)=\frac{3\sqrt 3}{2}$, $M'(1)=\frac{3\sqrt3}{4}$, see Figure \ref{fig-2F1}. Notice that the function $t\mapsto\frac{9}{2\sqrt3}(\frac{1+t^{2/3}}{2})^{3/2}$ (which one could guess based on \eqref{final-formula-2k-k} and \eqref{three-limits}) satisfies only the last two properties.
\begin{center}
\begin{figure}
\includegraphics[width=16cm]{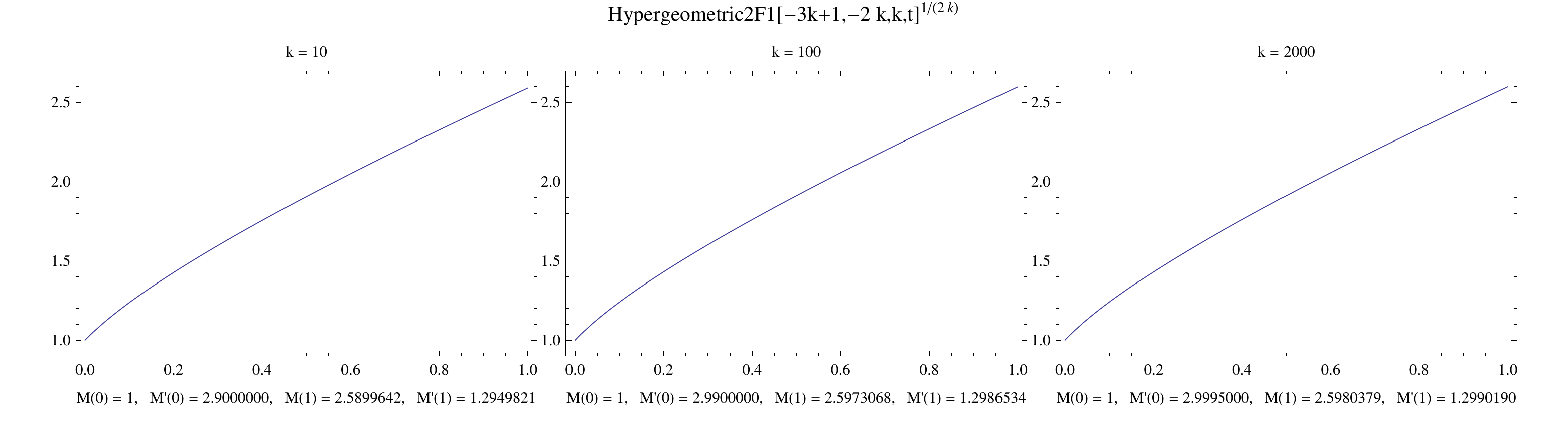}
\includegraphics[width=16cm]{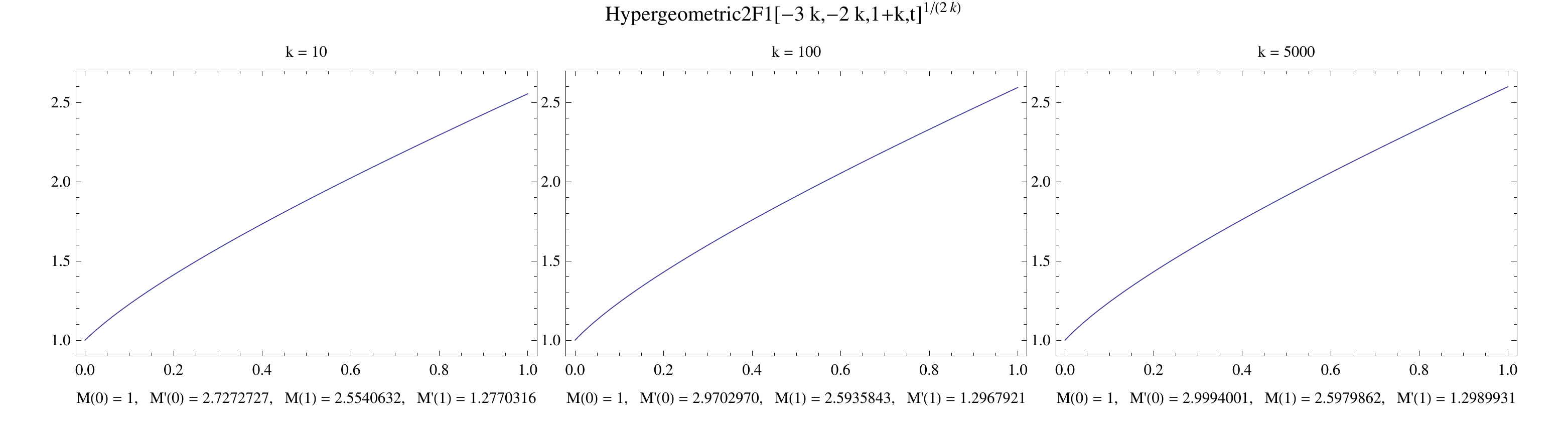}
\includegraphics[width=16cm]{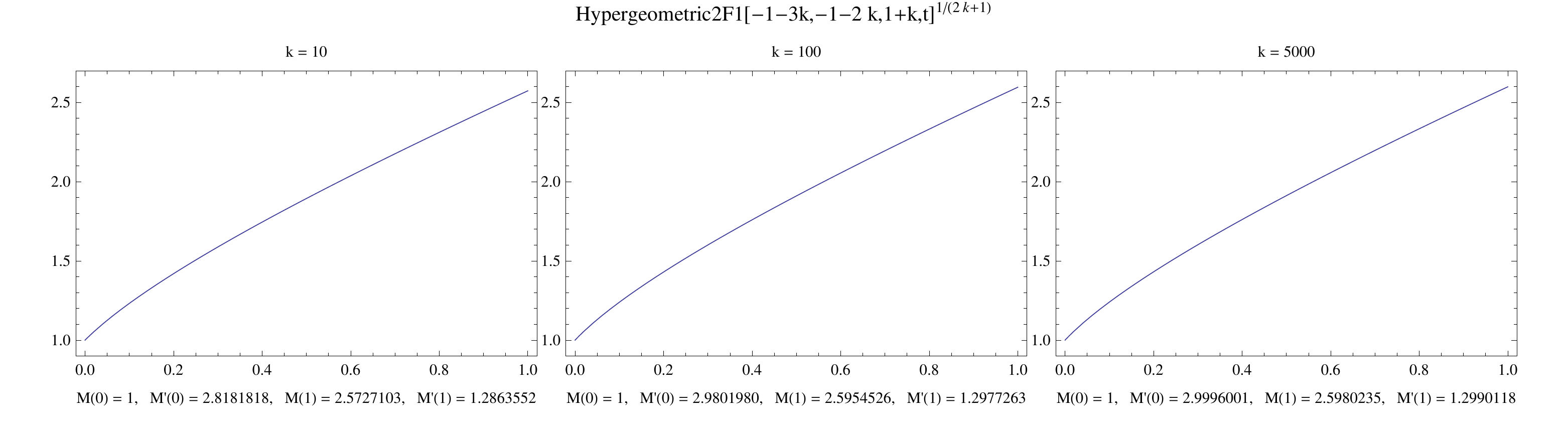}
\caption{The three functions in \eqref{2F1C} for $k \in \{10, 100, 5000\}$.}\label{fig-2F1}
\end{figure}
\end{center}

The above analysis supports the conjecture that for an arbitrary 2-periodic sequence $X=(x,y,x,y,\ldots)$ and every $0\leq c\leq 1$, the sequence $F_X(k,c)$ converges (not monotonically, unless $c=1/2$) to a limit.
We can repeat the above analysis for nonnegative sequences $X=(x_1,\ldots ,x_L,x_1,\ldots,x_L,\ldots)$ with longer period $L$, where
\be
F_X(k,c)=\frac{1}{{kL\choose\lceil ckL\rceil}}\sum_{j_1+\cdots+j_L=\lceil ckL\rceil} \prod_{l=1}^L {k\choose j_l}x_l^{j_l}
\ee See Figures \ref{threeFXkcL3} and \ref{nonmonotonicityL3} for a few examples with $L=3$. 
\begin{center}
\begin{figure}[h]
\includegraphics[width=16cm]{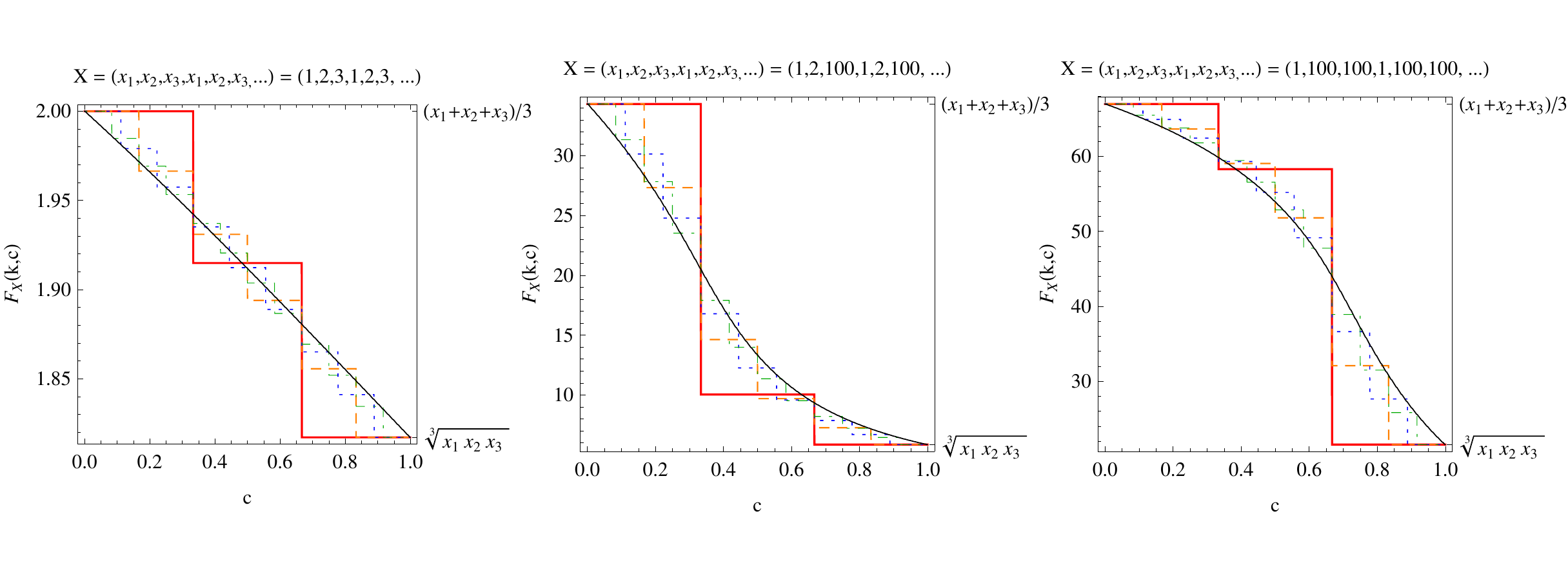}
\caption{The function $c\mapsto F_X(k,c)$ for three different $X$ of period $L=3$ and $k=1$ (solid red), $k=2$ (dashed orange), $k=3$ (dotted blue), $k=4$ (dash-dotted green), $k=200$ (solid black).}\label{threeFXkcL3}
\vspace{.5cm}
\includegraphics[width=16cm]{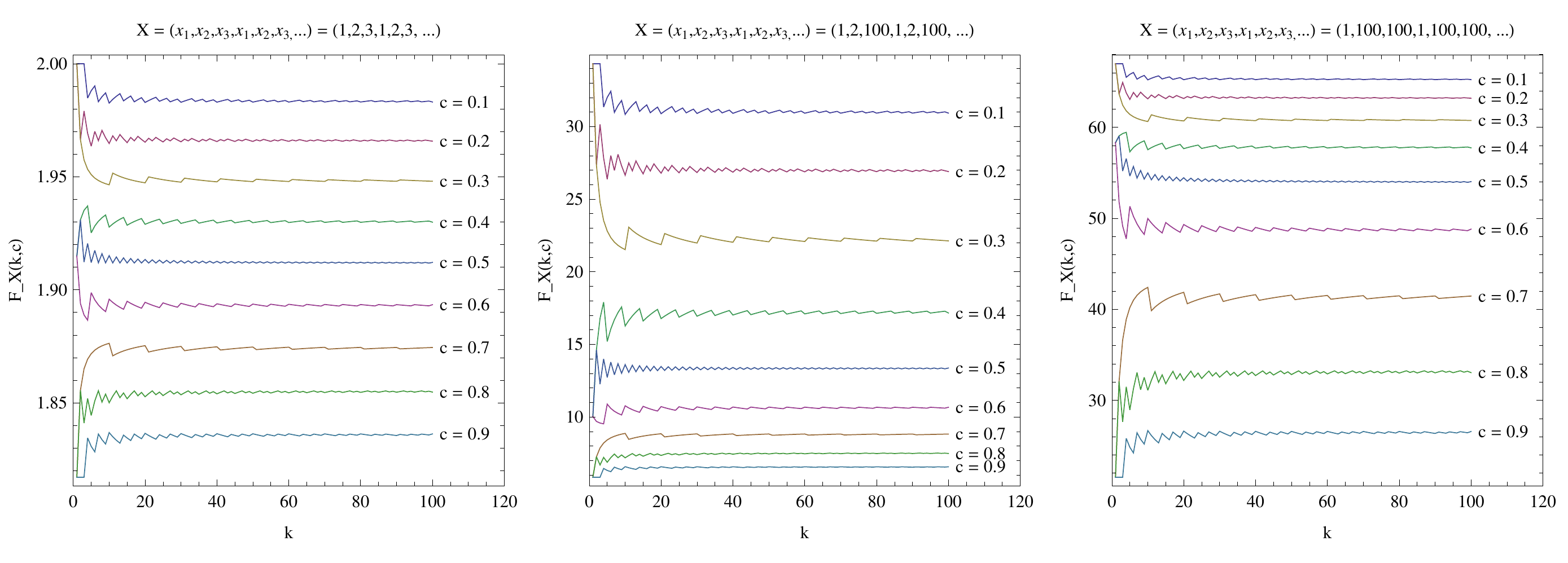}
\caption{Plot of the function $k\mapsto F_X(k,c)$ for the 3-periodic sequences $X$ in Figure \ref{threeFXkcL3} and for $c\in\{.1,.2,\ldots,.9\}$. }\label{nonmonotonicityL3}
\end{figure}
\end{center}

The above analysis allows us to formulate the following conjecture.

\begin{conj}\label{conj-periodic-c}
Let  $X=(x_1,x_2,\ldots, x_L,x_1,x_2,\ldots)$ be a periodic sequence of positive real numbers with finite period $L$. Then for any $c \in [0,1]$ the sequence $\{F_X(k,c)\}$ defined in \eqref{F_X(k,c)}  is convergent
and the limit
\be\label{def_F_X(c)}
 F_X(c) \ :=\ \lim_{k \to \infty} F_X(k,c)
\ee
is a continuous function of $c$.
\end{conj}

Notice that we already know that $F_X(0)=S(X,L,1)=(x_1+\cdots+x_L)/L$ and $F_X(1)=S(X,L,L)^{1/L}=\sqrt[L]{x_1\cdots x_L}$. Moreover, if the limit \eqref{def_F_X(c)} exists, then it is a decreasing function of $c$ by MacLaurin's inequalities \eqref{Maclaurin-inequalities}.

As pointed out already, the conjectured pointwise convergence of the sequence of functions $\{F_X(k,c)\}_{k\geq 1}$ to $F_X(c)$ is in general not monotonic in $k$. Despite this fact, a Dini-type theorem holds in this case since the limit function $c\mapsto F_X(c)$ is monotonic. We have the following.

\begin{prop}\label{Dini-like}
Assume Conjecture \ref{conj-periodic-c}. Then $\{F_X(k,c)\}_{k\geq 1}$ converges uniformly to $F_X(c)$ for $0\leq c\leq 1$ as $k\to\infty$.
\end{prop}
\begin{proof}
Fix $\varepsilon>0$. Choose $\{c_i\}_{i=1}^m\subset[0,1]$ such that $0=c_1<c_2<\cdots<c_m=1$ and $0\leq F_X(c_{i-1})-F_X(c_i)<\varepsilon$ for all $2\leq i\leq m$. Notice that this is possible if the distances between the $c_i$'s are small enough, since $c\mapsto F_X(c)$ is continuous.
Now, since $F_X(k,\cdot)$ converges pointwise to $F_X$ and $\{c_i\}_{i=1}^m$ is a finite set, we can choose $k$ large enough such that $|F_X(c_i)-F_X(k,c_i)|<\varepsilon$ for all $1\leq i\leq m$.
Consider an arbitrary $0\leq c\leq 1$. For some $1\leq i\leq m$ we have that $c_{i-1}\leq c\leq c_i$. Since $c\mapsto F_X(c)$ is non-increasing, we have
\be
F_X(k,c)\geq F_X(k,c_i)>F_X(c_i)+\varepsilon> F_X(c)+2\varepsilon.\nonumber
\ee
Similarly, we get
$
F_X(k,c)\leq F_X(k,c_{i-1})<F_X(c_{i-1})-\varepsilon<F_X(c)-2\varepsilon
$. Thus, for $k$ large enough, we obtain $|F_X(k,c)-F_X(c)|<2\varepsilon$ for every $0\leq c\leq 1$.
\end{proof}

If we assume Conjecture \ref{conj-periodic-c} (in which averages are taken over integral multiples of the period $L$), we can show that for every periodic sequence $X$ the averages $S(X,n,cn)^{1/cn}$ have a limit as $n\to\infty$ for every $0\leq c\leq 1$.

\begin{lem}
Let  $X=(x_1,x_2,\ldots, x_L,x_1,x_2,\ldots)$ be a periodic sequence of positive real numbers with finite period $L$. If we assume Conjecture \ref{conj-periodic-c} then for any $c \in [0,1]$, the limit
\be\lim_{n \to \infty} S(X,n,cn)^{1/cn} \ee
exists, and equals $F_X(c)$ defined in \eqref{def_F_X(c)}.
\end{lem}

\begin{proof}
Arguing as in the proof of Lemma \ref{difference-in-terms-S(alpha,n,cn)^1/cn}, we can show that there exists some constant $C$ such that
\be |S(X,n,cn)^{1/cn}  - S(X,n+1,c(n+1))^{1/c(n+1)}| \ \le \  \dfrac{C}{n}. \ee
Thus for any $n$ we can find a $k$ so that
\be |S(X,n,cn)^{1/cn}  - S(X,kL,ckL)^{1/ckL}| \ \le \  \dfrac{CL}{n}.\ee
However, by \eqref{def_F_X(c)}, the subsequence
\be \{S(X, kL, ckL)^{1/ckL}\}_{k\geq 1} \ee
converges to $F_X(c)$ as $k\to\infty$.
\end{proof}

\section{Approximating the averages for typical $\alpha$}\label{section-backtotypicalalpha}

In this section we provide a strengthening of Theorem \ref{thm1} assuming that Conjectures \ref{hyp1} and \ref{conj-periodic-c} are true.

\begin{thm}\label{thm2.7}
Assume Conjecture \ref{conj-periodic-c}. For any arithmetic function $f(n)$ which is $o(n)$ as $n\to\infty$, and almost all $\alpha$, we have
\be \limsup_{n \to \infty} S(\alpha, n, f(n))^{1/f(n)} \ = \  \infty. \ee
If we also assume Conjecture \ref{hyp1} we can replace the $\limsup$ with a limit.
\end{thm}

The proof of this theorem uses an approximation argument, where typical $\alpha$ are replaced by quadratic irrationals (discussed in Section \ref{section-periodic}) with increasing period. In the limit as the period tends to infinity, these numbers have same asymptotic frequency of continued fraction digits as typical real numbers.

To this end, recall that as discrete random variables, continued fraction digits are not independent (see \cite{Miller-TaklooBighash-book}).
However, for almost all $\alpha$ their limiting distribution is known to be the \emph{Gauss-Kuzmin distribution}:
\be \lim_{n \to \infty} \mathbb{P}[a_n(\alpha)= k] \ = \  \log_{2}{\left(1+ \frac{1}{k(k+2)}\right)}\ =:\ P_{\rm GK}(k).\ee

\begin{defi}For each integer $d > 1$ we define a periodic sequence $X_d$ via the following construction. For each $k \in \{2, 3, 4, \dots, d\}$ let $\lfloor P_{\rm GK}(k)\cdot10d^2 \rfloor$ of the first $10d^2$ digits of $X_d$ equal $k$, and set the remaining of the first $10d^2$ equal to 1. Extend $X_d$ so that it is periodic with period $10d^2$.
\end{defi}

We identify the periodic sequence $X_d$ with the corresponding continued fraction.
For $d=2$ we have $\lfloor P_{\rm GK}(2) \cdot 40\rfloor=6$ and
\begin{align}
X_2&\ =\ [\overline{2,2,2,2,2,2,\underbrace{1,1,\ldots,1}_{34}}]\ =\ \frac{-1457228823+5\sqrt{242075518250616389}}{2421016726}\nonumber\\
&\ \approx\ 0.4142184121;
\end{align}for $d=3$ we have $\lfloor P_{\rm GK}(2) \cdot 90\rfloor=15$, $\lfloor P_{\rm GK}(3) \cdot 90\rfloor=8$ and
\begin{align}X_3&\ = \ [\overline{\underbrace{2,\ldots,2}_{15},\underbrace{3,\ldots,3}_{8},\underbrace{1,\ldots,1}_{67}}]\approx0.4142135624;\end{align}
and so on. Note as $d\to\infty$ the digits $1,2,3,\ldots$ appear in $X_d$ with asymptotic frequencies $P_{\rm GK}(1), P_{\rm GK}(2), P_{\rm GK}(3),\dots$. The specific order of the digits does not matter since the symmetric means $S(X_d, k 10d^2, c k 10d^2)$ are invariant by permutation of the digits within each period. In particular, it is not relevant that $X_d\to \sqrt2-1=[\overline{2}]$ as $d\to\infty$

\begin{lem}\label{lem2.5}
Assume Conjecture \ref{conj-periodic-c}. For any $d >1$, $c \in (0,1]$, and almost all $\alpha$,
\be F_{X_d}(c) \ \le \    \limsup_{n \to \infty} S(\alpha, n, cn)^{1/cn}. \ee
\end{lem}

\begin{proof} Pick a subsequence $\{n_k\}$ of $\{n\}$ such that $S(\alpha,n_k,cn_k)^{1/cn_k}$ converges to the $\limsup$. For $n_k$ sufficiently large, at least $\lfloor P(j)n_k \rfloor$ of the first $n_k$ terms in $X(\alpha)$ are equal to $j$ for each $j \in \{2,3,\dots, d\}$. The desired inequality follows.
\end{proof}

\begin{lem}\label{lem2.6}
Assume Conjecture \ref{conj-periodic-c}. For any $M \in \mathbb{R}$ we can find  $c > 0$ sufficiently small and an integer $d$ sufficiently large such that
\be F_{X_d}(c)\ >\ M. \ee
\end{lem}

\begin{proof}
Since
 \be \displaystyle\sum_{k=1}^{d} \frac{k}{2} \log_{2}{\left(1 + \frac{1}{k(k+2)}\right)} \ee
 diverges as $d\to\infty$, we can pick a $d$ large enough so that $S(X_d, 10d^2, 1)$ is at least $2M$.  Then 
\be \lim_{c \to 0^+} F_{X_d}(c)\ =\ S(X_d, 10d^2, 1) \ \ge \  2M, \ee
and so for some $c > 0$ we must have $F_{X_d}(c) > M$.
\end{proof}

We can now use the lemmas above to prove Theorem \ref{thm2.7}.

\begin{proof}[Proof of Theorem \ref{thm2.7}] Suppose the $\limsup$ were equal to some finite number $M$ for some $f$ which is $o(n)$. Then simply let $d$ and $c$ be as in Lemma \ref{lem2.6}, and use Lemma \ref{lem2.5} to obtain a contradiction, since Maclaurin's inequalities \eqref{Maclaurin-inequalities} give us that
\be M \ < \   F_{X_d}(c) \ \le \    \limsup_{n \to \infty} S(\alpha, n, cn)^{1/cn} \ \le \  \limsup_{n \to \infty} S(\alpha, n, f(n))^{1/f(n)}. \ee
Assuming Conjecture \ref{hyp1}, we know
\be  \limsup_{n \to \infty} S(\alpha, n, cn)^{1/cn} \ = \   \liminf_{n \to \infty} S(\alpha, n, cn)^{1/cn} \ \le \  \liminf_{n \to \infty} S(\alpha, n, f(n))^{1/f(n)}, \ee
and thus we can say the limit is infinite in this case, since the $\liminf$ cannot be finite.
\end{proof}

We conclude this section with another conjecture, which states that the almost sure limit $\lim_{n\to\infty} S(\alpha,n,cn)^{1/cn}=F(c)$ (which exists if we assume Conjecture \ref{hyp1}), can be achieved by considering $\lim_{d\to\infty} F_{X_d}(c)$ (recall that $F_{X_d}(c)$ is well defined if we assume Conjecture \ref{conj-periodic-c}). The existence of the latter limit is proved in the following lemma.

\begin{lem}\label{lem2.8} Assume Conjecture \ref{conj-periodic-c}. Then for every $0\leq c\leq 1$ we have that $\lim_{d \to \infty} F_{X_d}(c)$ exists and is finite.
\end{lem}

\begin{proof}
Suppose that for some $c$ and some $d < d'$, we have $F_{X_d}(c) > F_{X_{d'}}(c)$. Then we can find a 
sufficiently large $N$ such that
\be S(X_d, (10Ndd')^2, c(10Ndd')^2)\ >\ S(X_{d'}, (10Nd d')^2, c(10Ndd')^2). \ee

However, if we rearrange the first $(10Ndd')^2$ terms of both $X_d$ and $X_{d'}$ and order them from least to greatest, we see from the definition of $X_d$ that this rearranged $X_{d'}$ is term by term greater than $X_d$, and so this is a contradiction. Thus
\be F_{X_d}(c) \ \le \  F_{X_{d'}}(c), \ee
and so by Lemma \ref{lem2.5} and the monotone convergence theorem, we get the existence of the limit and an upper bound:
\be \lim_{d \to \infty} F_{X_d}(c) \ \le \  K^{1/c}(K_{-1})^{1-\frac{1}{c}}. \ee
\end{proof}

As already anticipated, we conclude with a conjecture, which extends Conjecture \ref{hyp1}.

\begin{conj}\label{ex-conj-1}
For each $c \in (0,1]$  and almost all $\alpha$ the limit
$F(c)=\lim_{n\to\infty} S(\alpha,n,cn)^{1/cn}$ exists and
\be \lim_{d \to \infty} F_{X_d}(c) \ = \  F(c). \nonumber\ee
Moreover, the convergence is uniform in $c$ on compact subsets of $(0,1]$.
\end{conj}

\section{Proof of Theorem \ref{thm:maindough} and Corollary \ref{cor:maindough}}

\subsection{Preliminaries}
We begin with two lemmas about the tails of binomial distributions. As is well known, these approximate a bell curve and apart from a central section of a few standard deviations in width, there is little mass; our arguments require a quantitative version of this.

\begin{lem}\label{lem:binomialdhone} Let $m$, $n$, and $s$ be positive integers, with $m+s\le n$. Let $\lambda$ be real, with $0<\lambda\le 1/2$ and $\lambda n\le m$. Let $\sigma$ be positive, with $\sigma\le s/\sqrt{n\lambda(1-\lambda)}$. Then \begin{equation} \sum_{j=m+s}^n\binom{n}{j}\lambda^j(1-\lambda)^{n-j}\ \le\ e n \exp\left(-\frac{1}{2}\lambda\sigma^2\right).\end{equation}
\end{lem}

\begin{proof} Let $(t_0,t_1,\ldots,t_n)$ be the terms in the binomial expansion of $(\lambda+(1-\lambda))^n$. These of course sum to 1, and for $j\ge m$, $t_{j+1}<t_j$. Thus
\be \sum_{j=m+s}^nt_j\ \le\ nt_{m+s}\ \le\ nt_{m+s}/t_m.\ee
Now
\begin{align}
\frac{t_{m+s}}{t_m}&\ = \
\prod_{j=0}^{s-1}\frac{n-m-j}{m+j+1}\cdot\frac{\lambda}{1-\lambda}
\nonumber\\&\ \le\ \prod_{j=0}^{s-1}\frac{n(1-\lambda)-j}{n\lambda+j+1}\cdot\frac{\lambda}{1-\lambda}
\ = \ \prod_{j=0}^{s-1}\frac{1-j/(n(1-\lambda))}{1+(j+1)/(n\lambda)}.\end{align}
Thus \be\frac{t_{m+s}}{t_m}\ \le\ \prod_{j=0}^{s-1}\exp\left(-\frac{j}{n(1-\lambda)}\right)=
\exp\left(-\frac{s(s-1)}{2n(1-\lambda)}\right).\ee

Now $s<n$ and $1-\lambda\ge 1/2$ so $(1/2)s/(n(1-\lambda))<1$. Thus
\be\frac{t_{m+s}}{t_m}\ \le\ \exp\left(1-\frac{s^2}{2n(1-\lambda)}\right),\ee and the lemma now follows. \end{proof}

The companion lemma reads a little differently, and handles the other end of the summation.

\begin{lem}\label{lem:binomialdhtwo} Let $m$, $n$, and $s$ be positive integers with $m-s\ge 1$. Let $\lambda$ be positive, with $\lambda\le 1/2$ and $n\lambda\ge m$. Let $\sigma\le  s/\sqrt{n\lambda(1-\lambda)}$.
Then \be \sum_{j=0}^{m-s}\binom{n}{j}\lambda^j(1-\lambda)^{n-j}\ \le\ n e^{s/(2\lambda n)} \exp\left(-\frac{1}{2}(1-\lambda)\sigma^2\right).\ee
\end{lem}

\begin{proof}As before, $\sum_{j=0}^{m-s}t_j\le nt_{m-s}/t_m$. Now
\be\frac{t_{m-s}}{t_m}=\prod_{j=0}^{s-1}\frac{1-\lambda}{\lambda}\cdot\frac{m-j}{n-m+j+1}
\ \le\ \prod_{j=0}^{s-1}\left(1-\frac{j}{\lambda n}\right)\ \le\ \exp\left(-\frac{s(s-1)}{2\lambda n}\right).\ee
The result now follows.
\end{proof}

The purpose for these lemmas is to allow us to establish that when $l$ is small enough, which we shall specify as satisfying $2^l\le n^{1/4}$, the digits of $\alpha$ (or more accurately, proxies for them which we shall now describe) fall into various \emph{bins} with predictable frequency.

We now describe how proxy digits for $\alpha$ are constructed. The aim is to satisfy two conditions: first, that each proxy digit $\beta_i$ is within a factor of four of the actual digit $\alpha_i$ it replaces, and second, that each $\beta_i$ is independent of the rest of them, with ${\rm Prob}[\beta_i=2^l]=2^{-l}$ when $l$ is a positive integer. (The resulting $\beta_i$'s are, of course, not independent of the $\alpha_i$'s. Each of them is to some extent, strongly if $j=i$, and much more weakly as $|i-j|$ increases, correlated with all of the $\alpha_j$.)

The (original) digits $\alpha_i$ may be seen as \emph{random variables} on a probability space in which $X=[0,1]$ with the usual measure and sigma algebra. But there is another probability space that generates the same probabilities for any specification of a finite number of specific digits.

The underlying fact is that if we take $X_j$ to be the random variable determined by \be \alpha\ =\ \cfrac{1}{\alpha_1+\cfrac{1}{\alpha_2+\ddots+\frac{1}{\alpha_j+X_j}}}\ee then the conditional density for $X_j$ given the values for $\alpha_l$, $1\le l\le j$ has the form $(1+\theta_j)(1+\theta_jx)^{-2}$ on $[0,1]$, where $\theta_j$ is the finite reverse continued fraction $[\alpha_j,\ldots, \alpha_1]=1/(\alpha_j+1/(\alpha_{j-1}+\cdots+1/\alpha_1))$.

In this new space, then, there is no underlying $\alpha$. The set $Y=[0,1]^{\mathbb{Z}}$ takes the place of $X$, and with the usual measure where the cylinders are Cartesian products of measurable subsets of $[0,1]$. Elements of this probability space are thus sequences $(Y_1,Y_2,\ldots)$ of real numbers, which with probability 1 are all irrational.

We take $\alpha_1=d$ if $1/(d+1)<Y_1<1/d$. We take $X=Y$ and take $\beta_1=2^{\lceil -\log Y/\log 2\rceil}$.That way, if $1/2<Y<1$, then $\beta_1=2$, if $1/4<Y<1/2$, then $\beta_1=4$, and so on.

We next set $\theta=1/d$, set  $Y=Y_2$, and take $X_2$ so that \be \int_0^{X_2}f_{\theta}(x)\, dx\ = \ Y\ \Leftrightarrow\ \ Y \ = \ \frac{(1+\theta)X}{1+\theta X}.\ee We then take \be \alpha_2\ = \ \left\lfloor\frac{1}{X}\right\rfloor,\quad \beta_2 \ = \ 2^{\lceil-\log Y/\log 2\rceil}.\ee Since $Y$ is uniformly distributed in $[0,1]$, the probability that $\beta_2=2^{l}$ is $2^{-l}$.

Continuing in this vein, to determine $\alpha_j$ and $\beta_j$, we set $\theta=[\alpha_{j-1},\ldots,\alpha_1]$ and $Y=Y_j$. We choose $X$ so that $(1+\theta)X/(1+\theta X)=Y$, we take $\alpha_j=\lfloor 1/X\rfloor$, and $\beta_j=2^{\lceil-\log Y/\log 2\rceil}$. (For instance, if $Y_1=0.37$, $Y_2=0.19$, and $Y_3=0.88$, then $X_1=0.37$, so $\alpha_1=2$ because $1/3<X_1<1/2$. Since $\alpha_1=2$, $\theta_1=1/2$. Now from $Y_2=0.19$ we compute (to sufficient accuracy, because high precision is needed only to break ties) $X_2=.19/(1+(1/2)(1-.19))=.135$, and thus $1/X_2=1.405$ and $\alpha_2=7$. That makes $\theta_2=1/(7+1/2)=2/15$. Now $Y_3=0.88$ so $X_3=.88/(1+(2/15)(.12))=.866$, so $\alpha_3=1$ and $\theta_3=15/17$.
Meanwhile, directly from the $Y_j$'s, we have $\beta_1=4$, $\beta_2=8$, and $\beta_3=2$.)

We claim that (with probability 1) $\alpha_j/2<\beta_j<4\alpha_j$. The exclusion of sets of measure zero allows us to rule out equality in any of the bounds we have relating $X$, $Y$, $\alpha_j$, and $\beta_j$. For short, we write $\alpha$ in place of $\alpha_j$ here. (There is, in this model of the situation, no underlying $\alpha$ to generate the digits $\alpha_j$.) We write $\beta$ in place of $\beta_j$. First, we show that $\alpha/2<\beta$. Note that  $l>-\log Y/\log 2>l-1$, so that $1/\beta<Y<2/\beta$. Note also that we can write $X=1/(\alpha+\epsilon)$ with $0<\epsilon<1$.Since $Y=(1+\theta)X/(1+\theta X)$, this says that \be \frac{1}{\beta} \ < \ \frac{(1+\theta)/(\alpha+\epsilon)}{1+\theta/(\alpha+\epsilon)}\ < \ \frac{2}{\beta}.\ee
If $\alpha/2\ge \beta$, then \be \frac{2}{\alpha}\ \le\ \frac{1}{\beta}<\frac{1+\theta}{\alpha+\epsilon+\theta}.\ee Clearing fractions and simplifying, $2(\alpha+\epsilon+\theta)<\alpha+\epsilon\alpha$, which is impossible because $\theta<1$, $\alpha>0$, and $\epsilon>0$.

We next show that $\beta<4\alpha$. Suppose $\beta\ge 4\alpha$. Then $4/\beta\le 1/\alpha$, so $(1+\theta)/(\alpha+\epsilon+\theta)<1/2\alpha$. Clearing fractions, we have $2\alpha(1+\theta)<\alpha+\theta+\epsilon$, so $\alpha+\theta<\epsilon$, a contradiction.

\subsection{An Equivalent Theorem}

For purposes of Theorem \ref{thm:maindough}, the digits $\beta_j$ are perfect proxies for the digits $\alpha_j$. Any term contributing to  $S^{1/k}$ using the original digits is within a factor of 4 of the corresponding term using the proxy digits. The new $S$ is obtained by replacing each $\alpha_j$ with the corresponding $\beta_j$, but also the $\beta_j$ are independent (each from all the other $\beta_i$) and identically distributed, each taking value $2^l$ with probability $2^{-l}$. The following result immediately yields the upper bound in Theorem \ref{thm:maindough} as a corollary.

\begin{thm}\label{thm:maindoughequiv} There exist absolute, effectively computable positive constants $N$, $C_1$,  $C_2$, and $R$ with $C_1<C_2$ and $R>1$ such that for all $n\ge N$, for all $k$ with $n^{3/4}\le k\le n/R$,  with probability at least $1-n^{-4}$,\be C_1\log(n/k) \ \le\ S[(\beta_1,\ldots,\beta_n),n,k]^{1/k}\ \le\ C_2\log(n/k).\ee \end{thm}

\subsubsection{Upper Bound}

We now prove the upper bound in Theorem \ref{thm:maindoughequiv}.

\begin{proof} For an arbitrary positive integer $N$, the probability that a particular digit $\beta_j$ is as large as $N$ is at most $2/N$, so the probability that all of them are less than $N$ is at least  $1-2n/N$. Taking $N=n^6$, we discard all cases in which any digit is as large as $N$, while keeping most of the probability mass. The rest of the analysis assumes no large (greater than $n^6$) digits $\beta_j$. Now let $M=\lceil 6\log n/\log 2\rceil$, and let $B=(b_1,b_2,\ldots, b_M)$ be the list of the number of times, for $1\le l\le M$,  that a digit $\beta_j$ takes the value $2^l$.

For a list of $A$ of $M$ nonnegative integers, we say that $A\le B$ if $a_l\le b_l$ for $1\le l\le M$. With this notation, we have \be S[(\beta_1,\ldots,\beta_n),n,k]\ = \ \frac{1}{\binom{n}{k}}\sum_{A\le B}\prod_{l=1}^M\binom{b_l}{a_l}2^{2a_l}.\ee This is a key step. When many different values of $j$ correspond to the same $\beta_j$, the choice of subsets of $[n]$ resolves into a choice of \emph{how many} of the $b_l$ choices of $j$ for which $\beta_j=2^l$ we shall use, (that would be $a_l$), and then, \emph{which ones}. (There are $\tbinom{b_l}{a_l}$ ways to answer this second question.)

To analyze the likely behavior of this expression, we need again to discard improbable exceptional cases. Let $Q=\lceil(\log n-2\log\log n)/\log 2\rceil$. We now claim that if $2^l\ge n/\log^2n$, that is, if $l\ge Q$, then it is improbable that $b_l\ge \log^3 n$. This is quite plausible, since the expected value of $b_l$ (it is, we must keep in mind, a random variable) is $n/2^l\le \log^2n$. This requires another lemma.

\begin{lem} If $0<\gamma<1$ and $n\ge 1$ and $m\ge \gamma n$ then for $\tau>1$, \be\label{eq:eqinlemmarequiredbinom} \sum_{j=m}^n\binom{n}{j}\gamma^j(1-\gamma)^{n-j}\ <\ \tau^{-m}(1+\gamma(\tau-1))^n.\ee \end{lem}

\begin{proof}
The right side of \eqref{eq:eqinlemmarequiredbinom} is equal to $\sum_{j=0}^n\tbinom{n}{j}\gamma^j(1-\gamma)^{n-j}\tau^{j-m}$. The terms in which $j<m$ are at least positive, and they are competing with zero. The terms in which $j\ge m$ are the product  of the corresponding term on the left with $\tau^{j-m}$, which is at least 1. \end{proof}

Returning to the proof of Theorem \ref{thm:maindoughequiv}, we take $\gamma=2^{-l}$ and $m=\lceil \log^3 n\rceil$ and $\tau=\log n$ and conclude that when $l\ge Q$, \be {\rm Prob}[b_l\ge \log^3n] \ < \ (\log n)^{-\log^3n}\exp(\log^3n)\ = \ \exp(\log^3 n(1-\log\log n)).\ee For $n$ sufficiently large, this is much less than any particular negative integer power of $n$. We may safely discard digit strings in which $b_l\ge \log^3 n$ with $2^l\ge n/\log^2n$, and we do discard them.

Continuing with our program of expelling complicating exceptional cases, we now throw out all cases in which $2^l\le n/\log^2n$ and $b_l>2n/2^l$. (The expected value of $b_l$ is $n/2^l$ so getting twice as many as expected should be unlikely.) If $l=1$, it is outright impossible, so assume $l>1$. This time, we take $m=\lceil 2n/2^l\rceil$ and $\gamma=2^{-l}$ and $\tau=2$, and we conclude that \be {\rm Prob}[b_l\ge 2n/2^l] \ <\ 2^{-m}(1+2^{-l})^n \ < \ (e/4)^{n/2^l}\le(e/4)^{\log^2 n}.\ee Again, for $n$ sufficiently large, this is less than any particular negative power of $n$.

We subdivide the cases further. Let $r=n/k$ and let $L$ be the largest integer $l$ such that $2^l/l\le r$. This characterization of $L$ is needed but it takes a bit of calculation to get explicit bounds for $L$. Note that since $2^{L+1}/(L+1)>r$, for $r$ sufficiently large (and we choose $R$ so that this is assured) $L>\log r$. Thus $2^L>\tfrac{1}{2}r\log r$. On the other hand, if $L\ge (\log r+2\log\log r)/\log 2$, then $2^L/L\ge(r\log^2r)\log 2/(\log r+2\log\log r)$, which contradicts $2^L/L\le r$ when $r\ge R$ and $R$ is large enough  (since we control $R$, it is). Thus $\log r<L<(\log r+2\log\log r)/\log 2$. Another iteration of this kind yields $L>\log r/\log 2$.

What would happen to a term in $S$ if we converted all digits $\beta_j$ with $r\le \beta_j\le Q$ into 1's? That would reduce terms involving any such digit, but by a factor $D$ of at worst $\prod_{l=L+1}^Q 2^{2nl/2^l}$. The effect on $S$ is thus to reduce it by a factor $D$ satisfying  \be 1\ \le\ D \ \le\ 2^{2n\sum_{l=L+1}^Ql/2^l}\ \le\ 2^{4n(L+1)/2^{L+1}}\ \le\ 2^{4n/r}\ = \ 2^{4k}.\ee Taking the $1/k$ power of this, we see that such a replacement strategy can at worst reduce $S^{1/k}$ to $1/16$th of what it would otherwise have been.

As to the still larger digits, the effect of deleting them is to divide any term of $S$ by a factor $D'$ satisfying  \be 1\ \le\ D'\ \le\ \prod_{l=Q+1}^M2^{l\log^3n}\ <\ 2^{\log^3 n\sum_{l=Q}^Ml}\ <\ 2^{M^2\log^3 n}.\ee As $k\ge n^{3/4}$, we have $M^2\log^3n/k\rightarrow 0$ and  $D'^{1/k}$ tends to $1$ as $n\rightarrow\infty$.

This reduces the analysis down to the heart of the matter: \emph{Not counting the already controlled contributions from large, but infrequent, digits, and assuming the remaining values of $b_l$ are not too unusual, how large can $S$ be?}

We now apply Lemmas \ref{lem:binomialdhone} and \ref{lem:binomialdhtwo}. For $l$ with $2^l/l\le r$, we take $\lambda=2^{-l}$, $m=\lceil n/2^l\rceil$,  $s=\lceil (n/2^l)^{3/4}\rceil$, and $\sigma=s/\sqrt{\lambda(1-\lambda)}$ in Lemma \ref{lem:binomialdhone}. Writing $P={\rm Prob}[b_l\ge \frac{n}{2^l}+\left(\frac{n}{2^l}\right)^{3/4}+2]$,  we conclude that for $n$ sufficiently large,  \be P\ \le\ en\exp[-4n^{1/8}(\log n)^{-3/2}]<\exp[-n^{1/9}].\ee Similarly, in the other direction, the probability that $b_l$ falls short of $n/2^l$ by $(n/2^l)^{/34}+2$ is less than $\exp(-n^{1/4})$ for $n$ sufficiently large. Thus, for all $l$ with $1\le l\le L$, and for $n$ sufficiently large, $b_l$ is almost surely within $n/2^l\pm (n/2^l)^{3/4}$, give or take 1 or 2.

In the context of the theorem, big digits, that is, those greater than $2^L$, cannot affect the truth or falsity of the claim. There are (with very high probability) no more than $\log^3n$ digits greater than $n/\log^2n$, and none greater than $n^6$. Even if all of them somehow turned up in every term of $S$, they would not affect the result, because $(n^{6\log^3 n})\le 2^k$ for large $n$, and we can absorb factors such as $2^k$ simply by doubling $C$ in the statement of the theorem. The fairly big digits, the ones with $2^L< \beta\le n/\log^2n$, cannot affect the issue for similar reasons. They can at most contribute a factor of \be F\ = \ 2^{\sum_{l=L+1}^{\log n}2nl/2^l}.\ee Since $(L+1)/2^{L+1}<1/r$, $F<2^{16n/r}=2^{16k}$. As a result, we can with impunity reassign all large digits to any lesser value we please.

We set them all to 1. Since $2^L>r\log r/\log 2$, there are no more than $E=3n/(r\log r)+6\log^4n$ of them. At this point, what remains to be established is that with high probability, \be\label{eq:STAR} \frac{1}{\binom{n}{k}}\sum_{a_0=0^E}\binom{E}{a_0}\sum_{A\le B}\prod_{l=1}^L\binom{b_l}{a_l}2^{l a_l}\ \le\ (C\log r)^k\ee for suitably chosen $C$, where $B=(b_1,\ldots,b_L)$ and $A=(a_1,a_2,\ldots,a_L)$ with $0\le a_l\le b_l$ and $\sum_1^La_l=k-a_0$.

In \eqref{eq:STAR}, the effect of the first sum is at most a matter of multiplying the result of the largest second sum by $2^E$. Since $E<k$, this is harmless and it suffices to show that for any $kk$ between 0 and $k$ in place of $k$, the rest of the expression is bounded by some $(C\log r)^k$. As will become clear, the only case that matters is $kk=k$, so we now treat that case.

We need to prove that \be\label{eq:STAR2} \frac{1}{\binom{n}{k}}\sum_{A\le B}\prod_{l=1}^L\binom{b_l}{a_l}2^{l a_l}\ \le\ (C\log r)^k\ee for suitably chosen $C$, where $B=(b_1,\ldots,b_L)$ and $A=(a_1,a_2,\ldots,a_L)$ with $0\le a_l\le b_l$ and $\sum_1^La_l=k$. In this sum, we can safely replace $1/\binom{n}{k}$ with $k!/n^k$ since $k<n/2$ and we can absorb factors of $2^k$. We can safely replace $\binom{b_l}{a_l}$ with $b_l^{a_l}/a_l!$, for $1\le l\le L$, for the same reason. We can replace each $b_l$ with $n/2^l$ since $\prod_{1}^L(1+(2^l/n)^{-1/4})^{a_l}$ is safely small because $\sum a_l=k$.

At this point, we drop the condition that $A\le B$. The choices for $A$ are any list of $L$ nonnegative integers that sum to $k$. We claim that there exists $C>0$ such that for $n$ sufficiently large and $R\le n/k\le n^{1/4}$, \be \frac{k!}{n^k}\sum_{A}\prod_{l=1}^L\frac{(n/2^l)^{a_l}}{a_l!}2^{la_l}\ \le\ (C\log r)^k.\ee The powers of $n$ and of $2$ cancel, leaving us to prove \be \sum_{A}\frac{k!}{\prod_{l=1}^La_l!} \ \le\ (C\log r)^k;\ee however, this sum is exactly what one gets from expanding $(1+\cdots+ 1)^k$ ($L$ 1's added) according to the multinomial theorem. Therefore, the sum equals $L^k$, and with $2^{L+1}/(L+1)>r\ge 2^L/L$, it is clear that $L$ is comparable to $\log r$.

At this point, it is also clear that the upper bound we get in this fashion is larger than what we would get with any smaller value for the sum of the entries of $A$, as promised earlier. This completes the proof of Theorem \ref{thm:maindoughequiv} and with it, the upper bound for Theorem \ref{thm:maindough}. \end{proof}

\subsubsection{Lower Bound}

We now prove the lower bound in Theorem \ref{thm:maindoughequiv}.

\begin{proof} To find lower bounds for $S$, we can again discard unlikely events, and as a result, again work within the setting where $b_l$ is close to $n/2^l$ for $1\le l\le L$. We can of course demote large digits, should they occur, to values no greater than $2^L$, and we do. Our strategy for a lower bound is to pin all our hopes on a single term from $\sum_{A\le B}\prod_{l=1}^L\binom{b_l}{a_l}2^{la_l}$: the term in which all the $a_l$ are (as nearly as possible) equal. Since the sum of the $a_l$ is equal to $k$ (after demotions, if necessary) this means that each $a_l$ should be one of the integers bracketing $n/L$. We need a fact about factorials: for integers $b$ and $s$ with $b\ge 1$ and $1\le s\le b$, $\prod_{j=0}^{s-1}(b-j)>b^se^{-s}$. This follows by the integral comparison test, applied to $\int_0^s\log(b-x)\, dx$ and $\int_0^s\log(b-\lfloor x\rfloor)\, dx$.

Our list $(a_1,a_2,\ldots,a_L)$ has the form $a_l=k/L+\epsilon_l$, where for $1\le l\le L$, $|\epsilon_l|<1$, and where $\sum_{l=1}^L\epsilon_l=0$. The goal is to show that there exists $C_1>0$ so that $S\ge (C_1\log r)^k$ provided the values of $b_l$ fall within $\pm (n/2^l)^{3/4}$ for $1\le l\le L$. We are working with `binarized', independent digits $\beta_j$. We have \be S\ >\ \frac{1}{\binom{n}{k}}\prod_{l=1}^L\binom{b_l}{a_l}2^{la_l},\ee because the right side is just one of the terms of $S$. Thus \be S\ > \ \frac{k!}{n^k}\prod_{l=1}^L\frac{b_l^{a_l}e^{-a_l}}{a_l!}2^{la_l}\ee from our recent bounds on factorials.

We are working in the (highly probable) case that $b_l=n/2^l+\delta_l(n/2^l)^{3/4}$, with $|\delta_l|<1$ for $1\le l\le L$, so
\begin{align}
 S&\ > \ \frac{k!e^{-k}}{n^k}
 \prod_{l=1}^L\frac{1}{a_l!}\left(\frac{n}{2^l}\right)^{a_l}
 \left(1-\left(\frac{n}{2^2}\right)^{-1/4}\right)^{a_l}2^{la_l}
 \nonumber\\&\ = \ k!e^{-k}\prod_{l=1}^L\left(1-\left(\frac{n}{2^2}\right)^{-1/4}\right)^{a_l}
 \prod_{l=1}^L\frac{1}{a_l!}.\end{align} From our estimates for $L$ and the requirement that $2^{L+1}/(L+1)>r$, it follows that $2^{L+1}>r\log r/\log 2$. Thus $(n/2^l)^{-1/4}\le ((2\log r)/k)^{-1/4}$ and \be \prod_{l=1}^L\left(1-\left(\frac{n}{2^l}\right)^{-1/4}\right)^{a_l}\ \ge\ \left(1-\left(\frac{2\log r}{k}\right)^{1/4}\right)^k\ \ge\ 2^{-k}.\ee While there is a lot of slack in this step and we could avoid giving away the powers of 2, we don't need such savings.

We now have \begin{align}S& \ \ge\ \frac{k!}{(2e)^k}\prod_{l=1}^L\frac{1}{a_l!}
 \ge\frac{k!}{(2e)^k}\prod_{l=1}^L\left(\frac{k}{L}+\epsilon_l\right)^{-k/L+\epsilon_l}
\nonumber\\&
\ >\ \frac{k!}{(2e)^k}\left(\frac{k}{L}+1\right)^{-k}>\frac{k!}{(2e)^k}(2k)^{-k}L^k\nonumber\\
 &\ > \ \frac{k^ke^{-k}}{(2e)^k2^k}k^{-k}L^k\ > \ \frac{(\log r)^k}{(2e)^{2k}}.\end{align}
This completes the proof of the other direction of Theorem \ref{thm:maindoughequiv}, and thus of Theorem \ref{thm:maindough}.
\end{proof}

\subsection{A lower bound when $r$ is large}

For large $r$, we have

\begin{thm} There exist positive constants $\delta$, $C$ and $N$ such that for all $n\ge N$ and all $k$ with $1\le k\le n^{3/4}$, and with probability at least $1-\exp(-\delta\log^2 n)$, \be S(\alpha,n,k)\ge (C\log \lfloor n/k\rfloor)^k.\ee \end{thm}

\begin{proof} Let $r=\lfloor n/k\rfloor$. The basic idea here is that we cut up $[n]$ into $k$ intervals of length $r$, $[r]+t r$, $0\le t<k$, together with a possible rump interval of length less than $r$, which will not be used. We then restrict attention to terms of $S$ in which one of the $k$ digits $\alpha_j$ is taken from each of those intervals.

As we did earlier, we need to replace the original digit stream $(\alpha_j)$ of $\alpha$ with a new digit stream $(\beta_j)$ in such a way that each $\beta_j$ is (deterministically) within a constant multiple of the original corresponding $\alpha_j$, but so that also the $\beta_j$'s are, as random variables, independent of each other and identically distributed. The difference is that this time, that distribution has density function $u$ given by $u(x)=1/x^2$ for $x\ge 1$, and $0$ otherwise.

As before, we regard the digits $\alpha_j$ as being produced sequentially by a process regulated by an underlying sequence of probability density functions, each of the form $f_{\theta}(x)$ given by $f_{theta}(x)=(1+\theta)(1+\theta x)^{-2}$ if $0<x<1$, and by $0$ otherwise. Initially, $\theta_0=0$. If $\alpha_1,\ldots,\alpha_j$ have been chosen and it is time to `roll the dice' and see what $\alpha_{j+1}$ is, we set $\theta=\theta_j=[\alpha_j,\ldots,\alpha_1]=1/(\alpha_j+1/(\alpha_{j-1}+\cdots+1/\alpha_1)\cdots)$, we take a random real number $X_{j+1}$ chosen with density $f_{\theta}$ from $[0,1]$, and we take $\alpha_{j+1}=\lfloor 1/X_{j+1}\rfloor$. The choice of $\beta_j$ is driven by much of the same process, except that once we know $X$, we take $\beta=\beta_{j+1}$ so that \be \frac{1}{\beta}\ = \ \int_{x=\beta}^{\infty}x^{-2}\,dx \ = \ \int_{t=0}^Xf_{\theta}(t)\, dt=\frac{X(1+\theta)}{1+\theta X}.\ee The conditional density of $\beta=\beta_{j+1}$ given $\alpha_1,\ldots,\alpha_j$ and thus $\theta$, is in all cases $u(x)$. Thus the overall probability density function for $\beta_{j+1}$, being a weighted sum of the conditional densities, is also $u(x)$.

As to the relation between $\alpha=\alpha_{j+1}$ and $\beta=\beta_{j+1}$, boiled down, the preceding calculation gives $\beta=(1+\theta X)/(X+\theta X)$. If $1/(\alpha+1)<X\le 1/\alpha$, then \be (\alpha+\theta) / (1+\theta)\ <\ \beta \ < \ (\alpha+1+\theta)/(1+\theta),\ee so regardless of $\theta\in [0,1)$, $\alpha/2<\beta<2\alpha$. Thus using digits $\beta_j$ in place of $\alpha_j$ in calculating $S$ at worst reduces $S$ by a factor of $2^k$. This is acceptable, because we can just divide the `$C$' we get in the proof of the theorem under discussion but using $S$ determined with digits $\beta_j$ by $2$ for our result with respect to the original digits.

Now let $\mathcal A$ be the set of all subsets of $k$ elements of $[n]$ such that for each $j$ with $1\le j\le k$, exactly one element of $A'$ belongs to $\{(l-1)r+1,\ldots,(l-1)r+r\}$. We then have
\begin{align}S(\alpha,n,k)&\ \ge\ 2^{-k}S[(\beta_1,\ldots,\beta_n),n,k]\ =\ 2^{-k}\frac{1}{\binom{n}{k}}\sum_{A\in\mathcal A}\prod_{a\in A}\beta_a\nonumber\\&\ = \ 2^{-k}\frac{1}{\binom{n}{k}}\prod_{j=1}^k\sum_{l=1}^r\beta_{(j-1)r+l}\ \ge\ \frac{k^ke^{-k}}{2^k n^k}\prod_{j=1}^kU_j\end{align} where the $U_j$ are random variables, each identically distributed and independent of the others, with density $u_r$ that is the convolution of $r$ copies of $u$. (So that, for instance, $u_2(x)=0$ for $x<2$, and for $x>2$, $u_2(x)=\int_{y=2}^{\infty} y^{-2}(x-y)^{-2}\, dy$.) If we knew that $U$ was almost surely larger than $r\log r$, or even something in that ball park, we'd effectively be done.

It is known that the probability density functions $u_r$ converge in distribution, as $r\rightarrow\infty$, to appropriately scaled copies of the Landau density, one of a family of stable densities, and with the scaling taken into effect, very little of the mass of $u_r$ figures to sit substantially to the left of $r\log r$. The Landau distribution has a `fat tail' to the right, so that it is entirely possible that $U$ will be substantially larger than $r\log r$. All this, while informative, is not dispositive because the margin of error in the difference between $u_r$ and its limit is unfortunately large enough that we cannot use it in the proof of the result stated here.

Instead, we obtain an upper bound for the probability that $U<r\log r-Kr$ by studying the Laplace transform of $u$. For $s>0$,  let $F(s)=\int_1^{\infty}u(x)e^{-sx}$. Let $F_r(s)=\int_r^{\infty}u_r(x)e^{-sx}$. It is a well known property of the Laplace transform that it carries convolution to multiplication, so that, in particular, $F_r(s)=(F(s))^r$.

We now claim that for $0<s<1$,  $F(s)<\exp(s\log s)$. To see this, note that for $x\ge 1$ we have $e^{-sx}<1-sx+(1/2)s^2x^2$ since the series expansion of $e^{-sx}$ is alternating with terms of decreasing absolute value.  Thus \begin{align}F(s)&\ < \ \int_1^{1/s}x^{-2}(1-sx+\tfrac{1}{2}s^2x^2)\, dx+\int_{1/s}^{\infty}s^2 e^{-s x}\, dx\nonumber\\&\ =\ 1+s\log s-\left(\frac{1}{2}-\frac{1}{e}\right)s-\frac{1}{2}s^2<\exp(s\log s).\end{align}
Hence, $F_r(s)\le \exp(rs\log s)$.

Now for $K>0$ and $s>0$, \be {\rm Prob}[U\le r\log r-Kr\ = \ \int_{x=r}^{r\log r-Kr}u_r(x)\,dx\ < \ \int_{x=r}^{\infty}u_r(x)e^{s(r\log r-Kr-x)}\, dx.\ee We take $K=1+2\log\log r$ and $s=e^{K-1}/r$. Since $r\ge n^{1/4}$ and $n$ is large, $s<1$.

With our choice of $s$ and $K$,  after plugging in and simplifying we have \be {\rm Prob}[U\le r\log r-Kr]\ \le\ \exp[-\log^2 r]\ \le\ \exp[-\tfrac{1}{16}\log^2 n].\ee Thus with probability greater than $1-n\exp(-\log^2n/16)$, each of the $k$ $U'j$ is greater than $r\log r-Kr>\tfrac{1}{2}r\log r$. With high probability, we therefore have \be S\ \ge\ \frac{k^ke^{-k}}{4^kn^k}(r\log r)^k\ \ge\ \left(\frac{\log r}{5e}\right)^k,\ee this last bound using 5 instead of 4 in the denominator because $rk$ is perhaps a little less than $n$. This completes the proof.
\end{proof}

\subsection{Proof of Corollary \ref{cor:maindough}}

Armed with Theorem \ref{thm:maindough}, we show how Corollary \ref{cor:maindough} immediately follows.

First note that increasing $k$ decreases $S(\alpha,n,k)^{1/k}$. We thus begin by  replacing $f(n)$ with
$\max(n^{3/4}+1,f(n))$ so that Theorem \ref{thm:maindough} applies to $f$.

Write $k=k(n)$ for $\lfloor f(n)\rfloor$. Since $n/k\rightarrow\infty$ as $n\rightarrow\infty$,
$C_1\log(n/k)\rightarrow\infty$. Hence, for any $M>0$, there exists $N$ so that $C_1\log(n/k)>M$ for $n>N$, and thus for
$n>N$ we have $\text{Prob}\left[S^{1/k}(\alpha,n,k)\right]<n^{-4}$.

If $S^{1/k}(\alpha,n,k)$ does not tend to infinity then there exists an $M$ such that for all $N$ there exists $n>N$
with $S^{1/k}(\alpha,n,k)<M$. For $N$ large enough so that $C_1\log(n/k)>M$ for $n>N$, though, Theorem \ref{thm:maindough} implies that the  probability that there exists such an $n$ is less than $\sum_{n=N+1}^{\infty}n^{-4}<N^{-3}$. As the only number in $[0,1]$ that is
less than $N^{-3}$ for all $N$ is $0$, we see that with probability 1, $S^{1/k}(\alpha,n,k)\rightarrow\infty$. \hfill $\Box$

\appendix
\section{Computational Improvements}\label{sec:computationalimprovements}

We describe an alternative to the brute  force evaluation of $S(\alpha,n,k)$. In some rare cases (such as when the first $n$ digits of $\alpha$'s continued fraction expansion are distinct) there is no improvement in run-time; however, in general there are many digits repeated, and this repetition can be exploited. For example, the Gauss-Kuzmin theorem tells us that as $n\to\infty$ for almost all $\alpha$ we have approximately 41\% of the digits are 1's, about 17\% are 2's, about 9\% are 3's, about 6\% are 4's, and so on.

To compute $S(\alpha,n,k)$ we first construct the list \be L(\alpha,n)\ := \ (\alpha_1,\alpha_2,\ldots,\alpha_n)\ee of $\alpha$'s first $n$ continued fraction digits. Next, we set \be T(L(\alpha,n),n,k)\ = \ \sum_{A\subset \{1,\dots,n\} \atop |A|=k}\ \prod_{a\in A} \alpha_a,\ee and thus $S(\alpha,n,k)=T(L(\alpha,n),n,k)/\tbinom{n}{k}$.

Let $L'(\alpha,n)$ be the list of pairs $((m_1,d_1),(m_2,d_2),\dots,(m_u,d_u))$ where $d_1,d_2,\dots,d_u$ are the distinct digits that occur in $L(\alpha,n)$, and $m_1,m_2,\dots,m_u$ are their multiplicities. Thus $\sum_{j=1}^u m_j=n$, and we expect that typically $d_1=1$ with $m_1$ about $.41 n$, $d_2=2$ and $m_2$ is near $.17 n$, $d_3=3$ and $m_3$ around $.09 n$, and so on for a while (but not forever!).\footnote{We have noticed that the computations ran faster and used less memory when we wrote the digits in decreasing order, thus starting with the largest digit and going down to the 1's.} For instance, when $n=10$ and $\alpha=\pi-3$, we have $L(\pi-3,10)=(7, 15, 1, 292, 1, 1, 1, 2, 1, 3)$ and $L'(\pi-3,10)=((5,1),(1,2),(1,3),(1,7),(1,15),(1,292))$.

Now let $B(L'(\alpha,n))$ denote the set of all lists $b=(b_1,b_2,\dots,b_u)$ of $u$ non-negative integers that sum to $k$ and that satisfy $b_j\le m_j$ for $1\le j\le u$. For instance, with the example above if $k=3$ then one such $b$ would be $(2,0,0,0,1,0)$, and $B(L'(\alpha,n))$ has 26 elements in all.

It is not hard to see that \be T(L(\alpha,n),n,k)\ =\ \sum_{b\in B(L'(\alpha,n))}\prod_{j=1}^u\binom{m_j}{b_j}d_j^{b_j}.\ee This identity lends itself to a recursive algorithm which exploits the fact that all instances of a particular digit are the same and lumps them together by how many, rather than which specific ones, go into a particular product that contributes to $T$. For instance, with $n=2000$, $k=1000$ and $\alpha=\pi-3$ it takes less than ten seconds on the desktop of one of the authors to obtain $S^{1/k}$ numerically as 3.53672305321226. Done with the basic brute force algorithm, the same computation took 23 seconds. With $n=5000$ and $k=2500$, the corresponding calculation becomes out of reach with the basic algorithm. With the other approach, it required 35 seconds and reported that $S^{1/k}=3.5508312642208666735184$.


\bibliographystyle{plain}
\bibliography{ContinuedFractionsBibliography}

\ \\
\end{document}